\documentclass[11pt]{article}
\usepackage{amscd}
\usepackage{graphics}
\usepackage{color}
\usepackage{cancel}
\usepackage{stmaryrd}
\usepackage{tikz}
\usetikzlibrary{matrix,arrows}
\usepackage{amsmath}
\usepackage{amsthm}
\usepackage{amssymb}
\usepackage{proof}
\usepackage{verbatim,cmll}
\usepackage{setspace}
\usepackage{lscape}
\usepackage{latexsym,relsize}
\usepackage{amscd}
\usepackage{multicol}
\usepackage{bussproofs}
\usepackage{txfonts}
\usepackage{tikz}
\usepackage[position=top]{subfig}
\usepackage{leqno}
\usepackage{authblk}
\usepackage{marvosym}
\DeclareSymbolFont{extraup}{U}{zavm}{m}{n}
\DeclareMathSymbol{\varheart}{\mathalpha}{extraup}{86}
\DeclareMathSymbol{\vardiamond}{\mathalpha}{extraup}{87}
\DeclareMathSymbol{\vardiamond}{\mathalpha}{extraup}{87}

\newcommand{\commment}[1]{}

\renewcommand{\phi}{\varphi}
\renewcommand{\emptyset}{\varnothing}

\renewcommand{\epsilon}{\varepsilon}

\newcommand{\nomi}{\mathbf{i}}
\newcommand{\nomj}{\mathbf{j}}
\newcommand{\nomk}{\mathbf{k}}

\newcommand{\bigamp}{\mathop{\mbox{\Large \&}}}

\theoremstyle{plain}
\newtheorem{thm}{Theorem}
\newtheorem{theorem}{Theorem}[section]

\newtheorem{example}[theorem]{Example}

\newtheorem{proposition}[thm]{Proposition}

\newtheorem{lemma}[thm]{Lemma}

\theoremstyle{definition}

\newtheorem{definition}[thm]{Definition}

\title{Sahlqvist Correspondence Theory for Sabotage Modal Logic}
\author{Zhiguang Zhao}

\affil{\small School of Mathematics and Statistics, Taishan University, Tai'an, 271000, China}
\affil{\small zhaozhiguang23@gmail.com}
\date{}
\begin{document}
\maketitle
\begin{abstract}
\noindent

Sabotage modal logic (SML) \cite{vB05} is a kind of dynamic logic. It extends static modal logic with a dynamic modality which is interpreted as ``after deleting an arrow in the frame, the formula is true". In the present paper, we are aiming at solving an open problem stated in \cite{AuvBGr18}, namely giving a Sahlqvist-type correspondence theorem \cite{Sa75} for sabotage modal logic. In this paper, we define sabotage Sahlqvist formulas and give an algorithm $\mathsf{ALBA}^{\mathsf{SML}}$ to compute the first-order correspondents of sabotage Sahlqvist formulas. We give some remarks and future directions at the end of the paper.


\end{abstract}

\section{Introduction}

Sabotage modal logic (SML) \cite{vB05} belongs to the class of logics collectively called dynamic logics. It extends static modal logic with a dynamic modality $\Diamondblack$ such that $\Diamondblack\phi$ is interpreted as ``after deleting an arrow in the frame, $\phi$ is true". There are several existing works on sabotage modal logic. In \cite{AuvBGr18}, a bisimulation characterization theorem as well as a tableau system were given for sabotage modal logic, \cite{LoRo03a} proved the undecidability of the satisfiability problem and gave the complexity of the model-checking problem, and \cite{LoRo03b} gave the complexity of solving sabotage game. Several similar formalisms are also investigated, such as graph modifiers logic \cite{AuBadCHe09}, swap logic \cite{ArFeHo13} and arrow update logic \cite{KoRe11}, modal logic of definable link deletion \cite{Li18}, modal logic of stepwise removal \cite{vBMiZB19}. These logics are collectively called relation changing modal logics \cite{ArFeHo15,ArFeHoMa16}.
In the present paper, we are aiming at solving an open problem stated in \cite{AuvBGr18}, namely giving a Sahlqvist-type correspondence theorem \cite{Sa75} for sabotage modal logic. We define the Sahlqvist formulas in the sabotage modal language and give the sabotage counterpart of the algorithm $\mathsf{ALBA}^{\mathsf{SML}}$ (Ackermann Lemma Based Algorithm), which is sound over Kripke frames and is successful on  sabotage Sahlqvist formulas, to show that every Sahlqvist formula in the sabotage modal language has a first-order correspondent.

The structure of the paper is as follows: in Section \ref{Sec:Prelim}, we give a brief sketch on the preliminaries of sabotage modal logic, including its syntax, semantics as well as the standard translation. In Section \ref{Sec:UC} we define sabotage Sahlqvist formulas and the algorithm $\mathsf{ALBA}^{\mathsf{SML}}$, show its soundness over Kripke frames and its success on sabotage Sahlqvist inequalities. In Section \ref{Sec:discussion} we discuss the completeness and canonicity issues and give some further directions.

\section{Preliminaries on sabotage modal logic}\label{Sec:Prelim}

In this section, we collect some preliminaries on sabotage modal logic. For further details, see \cite{AuvBGr18}.

Given a set $\mathsf{Prop}$ of propositional variables, the set of sabotage modal formulas is recursively defined as follows:
$$\phi::=\bot \mid \top \mid p\mid \neg\phi \mid (\phi\land\phi) \mid \Diamond\phi \mid \Diamondblack\phi$$
where $p\in\mathsf{Prop}$, $\Diamond$ is the alethic connective of ordinary modal logic and $\Diamondblack$ is the sabotage connective of sabotage modal logic. We will follow the standard rules for omission of the parentheses. $\lor, \to, \leftrightarrow, \Box, \blacksquare$ can be defined in the standard way. We call a formula \emph{static} if it does not contain $\Diamondblack$ or $\blacksquare$\label{page:static}. An occurence of $p$ is said to be \emph{positive} (resp.\ \emph{negative}) in $\phi$ if $p$ is under the scope of an even (resp. odd) number of negations in the original sabotage modal language. A formula $\phi$ is positive (resp.\ negative) if all occurences of propositional variables in $\phi$ are positive (resp.\ negative).\label{page:positive:negative}

For the semantics of sabotage modal logic, we use Kripke frames $\mathbb{F}=(W,R)$ and Kripke models $\mathbb{M}=(\mathbb{F}, V)$ where $V:\mathsf{Prop}\to\mathcal{P}(W)$. The satisfaction relation is defined as follows:

\begin{center}
\begin{tabular}{l c l}
$(W, R, V),w\Vdash \bot$ & : & never;\\
$(W, R, V),w\Vdash \top$ & : & always;\\
$(W, R, V),w\Vdash p$ & iff & $w\in V(p)$;\\
$(W, R, V),w\Vdash \neg\varphi$ & iff & $(W, R, V),w\nVdash\varphi$;\\
$(W, R, V),w\Vdash\varphi\land\psi$ & iff & $(W,R,V),w\Vdash \varphi$ and $(W, R, V),w\Vdash\psi$;\\
$(W, R, V),w\Vdash\Diamond\varphi$ & iff & there exists a $v\in W$ such that $(w,v)\in R$ and $(W, R, V),v\Vdash\varphi$;\\
$(W, R, V), w\Vdash\Diamondblack\phi$ & iff & there exists an edge $(w_0,w_1)\in R \mbox{ s.t.\ }(W, R\setminus\{(w_0,w_1)\}, V), w\Vdash\phi.$\\
\end{tabular}
\end{center}

Intuitively, $\Diamondblack\phi$ is true at $w$ iff there is an edge $(w_0, w_1)$ of $R$ such that after deleting this edge from $R$, the formula $\phi$ is still true at $w$. It is easy to see that the semantic clause for $\blacksquare$ is defined as follows:
$$(W, R, V), w\Vdash\blacksquare\phi \mbox{ iff for all edges }(w_0,w_1)\in R, (W, R\setminus\{(w_0,w_1)\}, V), w\Vdash\phi.$$

The standard translation of sabotage modal language into first-order logic is given as follows (notice that we need to record the edges already deleted from $R$ so that we know what edges could still be deleted):

\begin{definition}(\cite[Definition 1]{AuvBGr18})
Let $E$ be a set of pairs $(y,z)$ of variables standing for edges and let $x$ be a designated variable. The translation $ST^{E}_{x}$ is recursively defined as follows:
\begin{itemize}
\item $ST^{E}_{x}(p):=Px$;
\item $ST^{E}_{x}(\bot):=x\neq x$;
\item $ST^{E}_{x}(\neg\phi):=\neg ST^{E}_{x}(\phi)$;
\item $ST^{E}_{x}(\phi\land\psi):=ST^{E}_{x}(\phi)\land ST^{E}_{x}(\psi)$;
\item $ST^{E}_{x}(\Diamond\phi):=\exists y(Rxy\land (\bigwedge\limits_{(v,w)\in E}\neg(x=v\land y=w))\land ST^{E}_{y}(\phi))$;
\item $ST^{E}_{x}(\Diamondblack\phi):=\exists y\exists z(Ryz\land (\bigwedge\limits_{(v,w)\in E}\neg(y=v\land z=w))\land ST^{E\cup\{(y,z)\}}_{x}(\phi))$.
\end{itemize}
\end{definition}

It is proved in \cite[Theorem 1]{AuvBGr18} that this translation is correct:

\begin{theorem}
For any pointed model $(\mathbb{M},w)$ and sabotage modal formula $\phi$, 
$$(\mathbb{M},w)\Vdash\phi\mbox{ iff }\mathbb{M}\vDash ST^{\emptyset}_{x}(\phi)[w].$$
\end{theorem}

\section{Algorithmic correspondence for sabotage modal logic}\label{Sec:UC}

In the present section, we will develop the correspondence algorithm $\mathsf{ALBA}^{\mathsf{SML}}$ for sabotage modal logic, in the style of \cite{CoGoVa06,CoPa12}. The basic idea is to use an algorithm $\mathsf{ALBA}^{\mathsf{SML}}$ to transform the modal formula $\phi(\vec p)$ into an equivalent set of pure quasi-(universally quantified) inequalities which does not contain occurrences of propositional variables, and therefore can be translated into the first-order correspondence language via the standard translation of the expanded language of sabotage modal logic (which will be defined on page \pageref{subsec:SML:expanded:lang}). 

The key ingredients of the algorithmic correspondence proof can be listed as follows:
\begin{itemize}
\item An expanded sabotage modal language as the syntax of the algorithm, as well as their interpretations in the relational semantics;
\item An algorithm $\mathsf{ALBA}$ which transforms a given sabotage modal formula $\phi(\vec p)$ into equivalent pure quasi-(universally quantified) inequalities $\mathsf{Pure}(\phi(\vec p))$;
\item A soundness proof of the algorithm;
\item A syntactically identified class of formulas on which the algorithm is successful;
\item A first-order correspondence language and first-order translation which transform pure quasi-(universally quantified) inequalities into their equivalent first-order correspondents.
\end{itemize}

In the remainder of the chapter, we will define an expanded sabotage modal language which the algorithm will manipulate (Section \ref{Subsub:expanded:language:SML}), define the first-order correspondence language of the expanded sabotage modal language and the standard translation (Section \ref{Subsub:FOL:ST}). We report on the definition of Sahlqvist inequalities (Section \ref{Subsec:Sahlqvist}), define a modified version of the algorithm $\mathsf{ALBA}^{\mathsf{SML}}$ (Section \ref{Subsec:ALBA}), and show its soundness (Section \ref{Subsec:soundness}) and success on Sahlqvist inequalities (Section \ref{Subsec:Success}).

\subsection{The expanded sabotage modal language $\mathcal{L}_{\blacksquare}^{+}$, the first-order correspondence language and the standard translation} \label{subsec:SML:expanded:lang}
\subsubsection{The expanded sabotage modal language $\mathcal{L}_{\blacksquare}^{+}$}\label{Subsub:expanded:language:SML}
In the present subsection, we give the definition of the expanded sabotage modal language $\mathcal{L}_{\blacksquare}^{+}$ and its standard translations, which will be used in the execution of the algorithm:
$$\phi::=p \mid \bot \mid \top \mid \nomi \mid \neg\phi \mid (\phi\land\phi) \mid (\phi\lor\phi) \mid (\phi\to\phi) \mid \Box\phi \mid \Diamond\phi \mid \blacksquare\phi \mid \Diamondblack\phi \mid $$
$$\Box^{S}\phi \mid \Diamond^{S}\phi \mid (\Box^{S})^{-1}\phi \mid (\Diamond^{S})^{-1}\phi \mid \mathsf{A}\phi \mid \mathsf{E}\phi \mid \forall\nomi\phi \mid \exists\nomi\phi$$

$$S::=\emptyset \mid S\cup\{(\nomi_{k0}, \nomi_{k1})\} $$

where $\nomi$ is called nominal as in hybrid logic \cite[Chapter 14]{BeBlWo06}, $\nomi_{k0}, \nomi_{k1}$ are fresh nominals not in $S$. We use the notation $\phi(\vec p)$ to indicate that the propositional variables occurring in $\phi$ are all in $\vec p$. We call a formula \emph{pure} if it does not contain propositional variables.

When interpreting the formulas in the expanded language, we assume that we start from a given pointed model $((W, R_0, V), w)$, and use $S$ to record the edges deleted from $R_0$. $\Box^{S}$, $\Diamond^{S}$ correspond to the relation $R_0\setminus\{(V(\nomi_{k0}), V(\nomi_{k1}))\mid (\nomi_{k0}, \nomi_{k1})\in S\}$ (denoted as $R_0\setminus S$), $(\Box^{S})^{-1}$, $(\Diamond^{S})^{-1}$ correspond to the relation $R_0^{-1}\setminus\{(V(\nomi_{k1}), V(\nomi_{k0}))\mid (\nomi_{k0}, \nomi_{k1})\in S\}$ (denoted as $(R_0\setminus S)^{-1}$), which intuitively means first delete the edges in S and then take the inverse relation. Unlabelled $\Box$ and $\Diamond$ correspond to the relation $R$ after certain deletions of edges. Therefore, we can say that $\Box^{S}$, $\Diamond^{S}$, $(\Box^{S})^{-1}$, $(\Diamond^{S})^{-1}$ are ``absolute connectives'' whose interpretations just depend on $R_0$ and $S$, while $\Box$ and $\Diamond$ are ``contextual connectives'' whose interpretations depend on the concrete $R$ after certain steps of deletions. For $\blacksquare$ and $\Diamondblack$, their interpretations depend on the context. $\mathsf{A}$ and $\mathsf{E}$ are global box and diamond modalities respectively, $(W, R, V),w\Vdash\forall\nomi\phi$ indicates that for all valuation variant $V^{\nomi}_{v}$ such that $V^{\nomi}_{v}$ is the same as $V$ except that $V^{\nomi}_{v}(\nomi)=\{v\}$, $(W, R, V^{\nomi}_{v}), w\Vdash\phi$, and $(W, R, V),w\Vdash\exists\nomi\phi$ is the corresponding existential statement.

For the semantics of the expanded sabotage modal language, the valuation is defined as  $V:\mathsf{Prop}\cup\mathsf{Nom}\to\mathcal{P}(W)$ similar to hybrid logic, and for the modal and dynamic connectives, the additional semantic clauses can be given as follows (notice that here $R$ is the ``actual'' accessibility relation in the model $(W, R, V)$ after some (maybe none) edges have been deleted, while $R_0$ is the ``starting accessibility'' relation when no edge has been deleted yet, and $R_0\setminus S$ is the notation for $R_0\setminus\{(V(\nomi_{k0}), V(\nomi_{k1}))\mid (\nomi_{k0}, \nomi_{k1})\in S\}$):

\begin{center}
\begin{tabular}{l c l}
$(W, R, V), w\Vdash \Box\phi$ & iff & for all $v$ s.t. $(w,v)\in R$, $(W, R, V), v\Vdash \phi$\\
$(W, R, V), w\Vdash \Diamond\phi$ & iff & there exists a $v$ s.t. $(w,v)\in R$ and $(W, R, V), v\Vdash \phi$\\
$(W, R, V), w\Vdash \blacksquare\phi$ & iff & for all edges $(w_0,w_1)\in R, (W, R\setminus\{(w_0,w_1)\}, V), w\Vdash\phi$\\
$(W, R, V), w\Vdash \Diamondblack\phi$ & iff & there exists an edge $(w_0,w_1)\in R$ s.t.\ $(W, R\setminus\{(w_0,w_1)\}, V), w\Vdash\phi$\\
$(W, R, V), w\Vdash \Box^{S}\phi$ & iff & for all $v$ s.t. $(w,v)\in(R_0\setminus S)$, $(W, R, V), v\Vdash \phi$\\
$(W, R, V), w\Vdash \Diamond^{S}\phi$ & iff & there exists a $v$ s.t.\  $(w,v)\in(R_0\setminus S)$ and $(W, R, V), v\Vdash \phi$\\
$(W, R, V), w\Vdash (\Box^{S})^{-1}\phi$ & iff & for all $v$ s.t. $(v,w)\in(R_0\setminus S)$, $(W, R, V), v\Vdash\phi$\\
$(W, R, V), w\Vdash (\Diamond^{S})^{-1}\phi$ & iff & there exists a $v$ s.t.\  $(v,w)\in(R_0\setminus S)$ and $(W, R, V), v\Vdash\phi$\\
$(W, R, V), w\Vdash \mathsf{A}\phi$ & iff & for all $v\in W$, $(W, R, V), v\Vdash\phi$\\
$(W, R, V), w\Vdash \mathsf{E}\phi$ & iff & there exists a $v\in W$ s.t.\ $(W, R, V), v\Vdash\phi$\\
$(W, R, V), w\Vdash\forall\nomi\phi$ & iff & for all $v\in W$, $(W, R, V^{\nomi}_{v}), w\Vdash\phi$\\
$(W, R, V), w\Vdash\exists\nomi\phi$ & iff & there exists a $v\in W$ s.t.\ $(W, R, V^{\nomi}_{v}), w\Vdash\phi$.\\
\end{tabular}
\end{center}

Here we do not require that each pair of nominals in $S$ denote different edges in $R_0$.

For the convenience of the algorithm, we introduce the following definitions:

\begin{definition}
$\ $
\begin{itemize}
\item An \emph{inequality} is of the form $\phi\leq^{S}_{S'}\psi$, where $\phi$ and $\psi$ are formulas, $S$ and $S'$ record the context of $\phi$ and $\psi$ respectively, i.e.\ which edges have already been deleted. Its interpretation is given as follows:
$$(W, R_0, V)\Vdash\phi\leq^{S}_{S'}\psi\mbox{ iff }$$$$(\mbox{for all }w\in W, \mbox{ if }(W, R_0\setminus S, V), w\Vdash\phi, \mbox{ then }(W, R_0\setminus S', V),w\Vdash\psi).$$

We use $\phi\leq\psi$ to denote $\phi\leq^{\emptyset}_{\emptyset}\psi$.

\item A \emph{quasi-inequality} is of the form $\phi_1\leq^{S_{1}}_{S'_{1}}\psi_1\ \&\ \ldots\ \&\ \phi_n\leq^{S_{n}}_{S'_{n}}\psi_n\ \Rightarrow\ \phi\leq^{S}_{S'}\psi$. Its interpretation is given as follows:
$$(W, R_0, V)\Vdash\phi_1\leq^{S_{1}}_{S'_{1}}\psi_1\ \&\ \ldots\ \&\ \phi_n\leq^{S_{n}}_{S'_{n}}\psi_n\ \Rightarrow\ \phi\leq^{S}_{S'}\psi\mbox{ iff }$$
$$(W, R_0, V)\Vdash\phi\leq^{S}_{S'}\psi\mbox{ holds whenever }(W, R_0, V)\Vdash\phi_i\leq^{S_i}_{S'_i}\psi_i\mbox{ for all }1\leq i\leq n.$$

\item A \emph{Mega-inequality} is defined inductively as follows:

$$\mathsf{Mega}::=\mathsf{Ineq}\mid \mathsf{Mega}\bigamp \mathsf{Mega} \mid \forall\nomi\forall\nomj(\nomi\leq^{S}_{S}\Diamond^{S}\nomj\Rightarrow \mathsf{Mega})$$

where $\mathsf{Ineq}$ is an inequality, $\bigamp$ is the meta-conjunction and $\Rightarrow$ is the meta-implication. Its interpretation is given as follows:

\begin{itemize}
\item $(W, R_0, V)\Vdash \mathsf{Ineq}$ iff the inequality holds as defined in the definition above;

\item $(W, R_0, V)\Vdash \mathsf{Mega_{1}}\bigamp\mathsf{Mega_{2}}$ iff $(W, R_0, V)\Vdash \mathsf{Mega_{1}}$ and $(W, R_0, V)\Vdash \mathsf{Mega_{2}}$;

\item $(W, R_0, V)\Vdash\forall\nomi\forall\nomj(\nomi\leq^{S}_{S}\Diamond^{S}\nomj\Rightarrow \mathsf{Mega})$ iff for all $w, v$, if $(w,v)\in (R_0\setminus S)$ then $(W,R_0,V^{\nomi,\nomj}_{w,v})\Vdash\mathsf{Mega}$, 
where $V^{\nomi,\nomj}_{w,v}$ is the same as $V$ except that $V^{\nomi,\nomj}_{w,v}(\nomi)=\{w\}$, $V^{\nomi,\nomj}_{w,v}(\nomj)=\{v\}$.

\end{itemize}
\item A \emph{universally quantified inequality} is defined as $\forall \nomi_1\ldots\forall\nomi_n(\phi\leq^{S}_{S'}\psi)$, and its interpretation is given as follows:

$(W, R_0, V)\Vdash\forall \nomi_1\ldots\forall\nomi_n(\phi\leq^{S}_{S'}\psi)$ iff for all $w_1, \ldots, w_n\in W$, $(W, R_0, V^{\nomi_1, \ldots, \nomi_n}_{w_1, \ldots, w_n})\Vdash\phi\leq^{S}_{S'}\psi$, where $V^{\nomi_1, \ldots, \nomi_n}_{w_1, \ldots, w_n}$ is the same as $V$ except that $V^{\nomi_1, \ldots, \nomi_n}_{w_1, \ldots, w_n}(\nomi_i)=\{w_i\}$, $i=1, \ldots, n$.

\item A \emph{quasi-universally quantified inequality} is defined as $\mathsf{UQIneq_1}\&\ldots\& \mathsf{UQIneq_n}\Rightarrow\mathsf{UQIneq}$ where $\mathsf{UQIneq}, \mathsf{UQIneq_i}$ are universally quantified inequalities. Its interpretation is given as follows:

$$(W, R_0, V)\Vdash\mathsf{UQIneq_1}\&\ldots\& \mathsf{UQIneq_n}\Rightarrow\mathsf{UQIneq}\mbox{ iff }$$
$$(W, R_0, V)\Vdash\mathsf{UQIneq}\mbox{ holds whenever }(W, R_0, V)\Vdash\mathsf{UQIneq_i}\mbox{ for all }1\leq i\leq n.$$

\end{itemize}
\end{definition}

It is easy to see that $(W, R_0, V)\Vdash\phi\leq^{\emptyset}_{\emptyset}\psi$ iff $(W, R_0, V)\Vdash\phi\to\psi$. We will find it easy to work with inequalities $\phi\leq\psi$ in place of implicative formulas $\phi\to\psi$ in Section \ref{Subsec:Sahlqvist}.

For inequalities, we have the following properties, which will be useful in the soundness proofs:
\begin{proposition}
\begin{itemize}
\item $(W,R_0,V)\Vdash\nomi\leq^{S}_{S}\Diamond^{S}\nomj$ iff $(V(\nomi), V(\nomj))\in (R_0\setminus S)$;
\item $(W,R_0,V)\Vdash\nomi\leq^{S}_{S'}\alpha$ iff $(W,(R_0\setminus S'),V), V(\nomi)\Vdash\alpha$, where $\alpha$ is a formula in the expanded sabotage modal language;
\item $(W,R_0,V)\Vdash A(\nomi\to\Diamond^{S}\nomj)$ iff $(V(\nomi), V(\nomj))\in (R_{0}\setminus S)$.
\end{itemize}
\end{proposition}

\subsubsection{The first-order correspondence language and the standard translation}\label{Subsub:FOL:ST}

In the first-order correspondence language, we have a binary relation symbol $R$ corresponding to the binary relation, a set of constant symbols $i$ corresponding to each nominal $\nomi$, a set of unary predicate symbols $P$ corresponding to each propositional variable $p$. 

The standard translation of the expanded sabotage modal language is defined as follows:
\begin{definition}
Let $E$ be a finite set of pairs $(y,z)$ of variables standing for edges and let $x$ be a designated variable. The translation $ST^{E}_{x}$ is recursively defined as follows:

\begin{itemize}
\item $ST^{E}_{x}(\bot):=x\neq x$;
\item $ST^{E}_{x}(\top):=x=x$;
\item $ST^{E}_{x}(\nomi):=x=i$;
\item $ST^{E}_{x}(p):=Px$;
\item $ST^{E}_{x}(\neg\phi):=\neg ST^{E}_{x}(\phi)$;
\item $ST^{E}_{x}(\phi\land\psi):=ST^{E}_{x}(\phi)\land ST^{E}_{x}(\psi)$;
\item $ST^{E}_{x}(\phi\lor\psi):=ST^{E}_{x}(\phi)\lor ST^{E}_{x}(\psi)$;
\item $ST^{E}_{x}(\phi\to\psi):=ST^{E}_{x}(\phi)\to ST^{E}_{x}(\psi)$;
\item $ST^{E}_{x}(\Box\phi):=\forall y(Rxy\land (\bigwedge\limits_{(v,w)\in E}\neg(x=v\land y=w))\to ST^{E}_{y}(\phi))$;
\item $ST^{E}_{x}(\Diamond\phi):=\exists y(Rxy\land (\bigwedge\limits_{(v,w)\in E}\neg(x=v\land y=w))\land ST^{E}_{y}(\phi))$;
\item $ST^{E}_{x}(\blacksquare\phi):=\forall y\forall z(Ryz\land (\bigwedge\limits_{(v,w)\in E}\neg(y=v\land z=w))\to ST^{E\cup\{(y,z)\}}_{x}(\phi))$;
\item $ST^{E}_{x}(\Diamondblack\phi):=\exists y\exists z(Ryz\land (\bigwedge\limits_{(v,w)\in E}\neg(y=v\land z=w))\land ST^{E\cup\{(y,z)\}}_{x}(\phi))$;
\item $ST^{E}_{x}(\Box^{S}\phi):=\forall y(Rxy\land (\bigwedge\limits_{(\nomi_{k0}, \nomi_{k1})\in S}\neg(x=i_{k0}\land y=i_{k1}))\to ST^{E}_{y}(\phi))$;
\item $ST^{E}_{x}(\Diamond^{S}\phi):=\exists y(Rxy\land (\bigwedge\limits_{(\nomi_{k0}, \nomi_{k1})\in S}\neg(x=i_{k0}\land y=i_{k1}))\land ST^{E}_{y}(\phi))$;
\item $ST^{E}_{x}((\Box^{S})^{-1}\phi):=\forall y(Ryx\land (\bigwedge\limits_{(\nomi_{k0}, \nomi_{k1})\in S}\neg(y=i_{k0}\land x=i_{k1}))\to ST^{E}_{y}(\phi))$;
\item $ST^{E}_{x}((\Diamond^{S})^{-1}\phi):=\exists y(Ryx\land (\bigwedge\limits_{(\nomi_{k0}, \nomi_{k1})\in S}\neg(y=i_{k0}\land x=i_{k1}))\land ST^{E}_{y}(\phi))$.
\item $ST^{E}_{x}(\mathsf{A}\phi):=\forall y ST^{E}_{y}(\phi)$;
\item $ST^{E}_{x}(\mathsf{E}\phi):=\exists y ST^{E}_{y}(\phi)$;
\item $ST^{E}_{x}(\forall\nomi\phi):=\forall i ST^{E}_{x}(\phi)$;
\item $ST^{E}_{x}(\exists\nomi\phi):=\exists i ST^{E}_{x}(\phi)$.
\end{itemize}
\end{definition}

It is easy to see that this translation is correct:

\begin{proposition}
For any pointed model $(\mathbb{M},w)$ and sabotage modal formula $\phi$, 
$$(\mathbb{M},w)\Vdash\phi\mbox{ iff }\mathbb{M}\vDash ST^{\emptyset}_{x}(\phi)[w].$$
\end{proposition}

For inequalities, quasi-inequalities, mega-inequalities, universally quantified inequalities and quasi-universally quantified inequalities, the standard translation is given in a global way:

\begin{definition}

\begin{itemize}
\item $ST(\phi\leq^{S}_{S'}\psi):=\forall x(ST^{S}_{x}(\phi)\to ST^{S'}_{x}(\psi))$;
\item $ST(\phi_1\leq^{S_{1}}_{S'_{1}}\psi_1\&\ldots\&\phi_n\leq^{S_{n}}_{S'_{n}}\psi_n\Rightarrow\phi\leq^{S}_{S'}\psi):= ST(\phi_1\leq^{S_1}_{S'_1}\psi_1)\land\ldots\land ST(\phi_n\leq^{S_n}_{S'_n}\psi_n)\to ST(\phi\leq^{S}_{S'}\psi)$;
\item $ST(\mathsf{Mega}_1\ \&\ \mathsf{Mega}_2)=ST(\mathsf{Mega}_1)\land ST(\mathsf{Mega}_2)$;
\item $ST(\forall\nomi\forall\nomj(\nomi\leq^{S}_{S}\Diamond^{S}\nomj\Rightarrow \mathsf{Mega})):=\forall i\forall j(Rij\land(\bigwedge\limits_{(v,w)\in S}\neg(i=v\land j=w))\to ST(\mathsf{Mega}))$;
\item $ST(\forall\nomi_1\ldots\forall\nomi_n\mathsf{Ineq}):=\forall i_1\ldots\forall i_n ST(\mathsf{Ineq})$;
\item $ST(\mathsf{UQIneq_1}\&\ldots\&\mathsf{UQIneq_n}\Rightarrow\mathsf{UQIneq}):= ST(\mathsf{UQIneq_1})\land\ldots\land ST(\mathsf{UQIneq_n})\to ST(\mathsf{UQIneq})$.
\end{itemize}
\end{definition}

\begin{proposition}\label{Prop:ST:ineq:quasi:mega}
For any model $\mathbb{M}$ and inequality $\mathsf{Ineq}$, quasi-inequality $\mathsf{Quasi}$, mega-inequality $\mathsf{Mega}$, universally quantified inequality $\mathsf{UQIneq}$, quasi-universally quantified inequality $\mathsf{QUQIneq}$,

$$\mathbb{M}\Vdash\mathsf{Ineq}\mbox{ iff }\mathbb{M}\vDash ST(\mathsf{Ineq});$$
$$\mathbb{M}\Vdash\mathsf{Quasi}\mbox{ iff }\mathbb{M}\vDash ST(\mathsf{Quasi});$$
$$\mathbb{M}\Vdash\mathsf{Mega}\mbox{ iff }\mathbb{M}\vDash ST(\mathsf{Mega});$$
$$\mathbb{M}\Vdash\mathsf{UQIneq}\mbox{ iff }\mathbb{M}\vDash ST(\mathsf{UQIneq});$$
$$\mathbb{M}\Vdash\mathsf{QUQIneq}\mbox{ iff }\mathbb{M}\vDash ST(\mathsf{QUQIneq}).$$
\end{proposition}

\subsection{Sahlqvist inequalities}\label{Subsec:Sahlqvist}

In the present section, since we will use the algorithm $\mathsf{ALBA}^{\mathsf{SML}}$ which is based on the classsification of nodes in the signed generation trees of sabotage modal formulas, we will use the unified correspondence style definition (cf.\ \cite{CPSZ,CPZ:Trans,PaSoZh16}) to define Sahlqvist inequalities. We will collect all the necessary preliminaries on Sahlqvist formulas/inequalities. 

\begin{definition}[Order-type of propositional variables](cf.\ \cite[page 346]{CoPa12})
For an $n$-tuple $(p_1, \ldots, p_n)$ of propositional variables, an order-type $\epsilon$ of $(p_1, \ldots, p_n)$ is an element in $\{1,\partial\}^{n}$. In the order-type $\epsilon$, we say that $p_i$ has order-type 1 if $\epsilon_i=1$, and denote $\epsilon(p_i)=1$ or $\epsilon(i)=1$; we say that $p_i$ has order-type $\partial$ if $\epsilon_i=\partial$, and denote $\epsilon(p_i)=\partial$ or $\epsilon(i)=\partial$.
\end{definition}

\begin{definition}[Signed generation tree]\label{adef: signed gen tree}(cf.\ \cite[Definition 4]{CPZ:Trans})
The \emph{positive} (resp.\ \emph{negative}) {\em generation tree} of any given formula $\phi$ is defined by first labelling the root of the generation tree of $\phi$ with $+$ (resp.\ $-$) and then labelling the children nodes as follows:
\begin{itemize}
\item Assign the same sign to the children nodes of any node labelled with $\lor, \land, \Box$, $\Diamond$, $\blacksquare, \Diamondblack, \Box^{S}, \Diamond^{S}, (\Box^{S})^{-1}, (\Diamond^{S})^{-1},\mathsf{A},\mathsf{E},\forall\nomi, \exists\nomi$;
\item Assign the opposite sign to the child node of any node labelled with $\neg$;
\item Assign the opposite sign to the first child node and the same sign to the second child node of any node labelled with $\to$.
\end{itemize}
Nodes in signed generation trees are \emph{positive} (resp.\ \emph{negative}) if they are signed $+$ (resp.\ $-$).
\end{definition}

Signed generation trees will be used in the inequalities $\phi\leq\psi$, where the positive generation tree $+\phi$ and the negative generation tree $-\psi$ will be considered. We will also say that an inequality $\phi\leq\psi$ is \emph{uniform} in a variable $p_i$ if all occurrences of $p_i$ in $+\phi$ and $-\psi$ have the same sign, and that $\phi\leq\psi$ is $\epsilon$-\emph{uniform} in an array $\vec{p}$ if $\phi\leq\psi$ is uniform in $p_i$, occurring with the sign indicated by $\epsilon$ (i.e., $p_i$ has the sign $+$ if $\epsilon(p_i)=1$, and has the sign $-$ if $\epsilon(p_i)=\partial$), for each propositional variable $p_i$ in $\vec{p}$.

For any given formula $\phi(p_1,\ldots p_n)$, any order-type $\epsilon$ over $n$, and any $1 \leq i \leq n$, an \emph{$\epsilon$-critical node} in a signed generation tree $\ast\phi$ (where $\ast\in\{+,-\}$) is a leaf node $+p_i$ when $\epsilon_i = 1$ or $-p_i$ when $\epsilon_i = \partial$. An $\epsilon$-{\em critical branch} in a signed generation tree is a branch from an $\epsilon$-critical node. The $\epsilon$-critical occurrences are intended to be those which the algorithm $\mathsf{ALBA}^{\mathsf{SML}}$ will solve for. We say that $+\phi$ (resp.\ $-\phi$) {\em agrees with} $\epsilon$, and write $\epsilon(+\phi)$ (resp.\ $\epsilon(-\phi)$), if every leaf node in the signed generation tree of $+\phi$ (resp.\ $-\phi$) is $\epsilon$-critical.

We will also use the notation $+\psi\prec \ast \phi$ (resp.\ $-\psi\prec \ast \phi$) to indicate that an occurence of a subformula $\psi$ inherits the positive (resp.\ negative) sign from the signed generation tree $\ast \phi$, where $\ast\in\{+,-\}$. We will write $\epsilon(\gamma) \prec \ast \phi$ (resp.\ $\epsilon^\partial(\gamma) \prec \ast \phi$) to indicate that the signed generation subtree $\gamma$, with the sign inherited from $\ast \phi$, agrees with $\epsilon$ (resp.\ with $\epsilon^\partial$). We say that a propositional variable $p$ is \emph{positive} (resp.\ \emph{negative}) in $\phi$ if $+p\prec+\phi$ (resp.\ $-p\prec+\phi$).\label{page:epsilon:subtree}

\begin{definition}\label{adef:good:branches}(cf.\ \cite[Definition 5]{CPZ:Trans})
Nodes in signed generation trees are called \emph{outer nodes} and \emph{inner nodes}, according to Table \ref{aJoin:and:Meet:Friendly:Table}.

A branch in a signed generation tree is called a \emph{excellent branch} if it is the concatenation of two paths $P_1$ and $P_2$, one of which might be of length $0$, such that $P_1$ is a path from the leaf consisting (apart from variable nodes) of inner nodes only, and $P_2$ consists (apart from variable nodes) of outer nodes only.
\begin{table}
\begin{center}
\begin{tabular}{| c | c |}
\hline
Outer &Inner\\
\hline
\begin{tabular}{c c c c c c c c c c }
$+$ & $\vee$ & $\wedge$ &$\Diamond$ & $\Diamondblack$ & $\neg$\\
$-$ & $\wedge$ & $\vee$ &$\Box$ & $\blacksquare$ & $\neg$ & $\to$\\
\end{tabular}
&
\begin{tabular}{c c c c c}
$+$ &$\wedge$ &$\Box$ & $\blacksquare$ & $\neg$ \\
$-$ &$\vee$ &$\Diamond$ & $\Diamondblack$ & $\neg$ \\
\end{tabular}\\
\hline
\end{tabular}
\end{center}
\caption{Outer and Inner nodes.}\label{aJoin:and:Meet:Friendly:Table}
\vspace{-1em}
\end{table}
\end{definition}

\begin{definition}[Sahlqvist and inequalities]\label{aInducive:Ineq:Def}(cf.\ \cite[Definition 6]{CPZ:Trans})
For any order-type $\epsilon$, the signed generation tree $*\phi$ of a formula $\phi(p_1,\ldots p_n)$ is \emph{$\epsilon$-Sahlqvist} if for all $1 \leq i \leq n$, every $\epsilon$-critical branch with leaf $p_i$ is excellent. An inequality $\phi\leq\psi$ is \emph{$\epsilon$-Sahlqvist} if the signed generation trees $+\phi$ and $-\psi$ are $\epsilon$-Sahlqvist. An inequality $\phi\leq\psi$ is \emph{Sahlqvist} if it is \emph{$\epsilon$}-Sahlqvist) for some $\epsilon$.
\end{definition}

\subsection{The algorithm $\mathsf{ALBA}^{\mathsf{SML}}$ for the sabotage modal language}\label{Subsec:ALBA}

In the present section, we define the correspondence algorithm $\mathsf{ALBA}^{\mathsf{SML}}$ for sabotage modal logic, in the style of \cite{CoGoVa06,CoPa12}. The algorithm goes in four steps. 

\begin{enumerate}

\item \textbf{Preprocessing and first approximation}:

In the generation tree of $+\phi$ and $-\psi$\footnote{The discussion below relies on the definition of signed generation tree in Section \ref{Subsec:Sahlqvist}. In what follows, we identify a formula with its signed generation tree.},

\begin{enumerate}
\item Apply the distribution rules:

\begin{enumerate}
\item Push down $+\Diamond, +\Diamondblack, -\neg, +\land, -\to$ by distributing them over nodes labelled with $+\lor$ which are outer nodes, and

\item Push down $-\Box, -\blacksquare,+\neg, -\lor, -\to$ by distributing them over nodes labelled with $-\land$ which are outer nodes.

\end{enumerate}

\item Apply the splitting rules:

$$\infer{\alpha\leq\beta\ \ \ \alpha\leq\gamma}{\alpha\leq\beta\land\gamma}
\qquad
\infer{\alpha\leq\gamma\ \ \ \beta\leq\gamma}{\alpha\lor\beta\leq\gamma}
$$

\item Apply the monotone and antitone variable-elimination rules:

$$\infer{\alpha(\perp)\leq\beta(\perp)}{\alpha(p)\leq\beta(p)}
\qquad
\infer{\beta(\top)\leq\alpha(\top)}{\beta(p)\leq\alpha(p)}
$$

for $\beta(p)$ positive in $p$ and $\alpha(p)$ negative in $p$.

\end{enumerate}

We denote by $\mathsf{Preprocess}(\phi\leq\psi)$ the finite set $\{\phi_i\leq\psi_i\}_{i\in I}$ of inequalities obtained after the exhaustive application of the previous rules. Then we apply the following first approximation rule to every inequality in $\mathsf{Preprocess}(\phi\leq\psi)$:

$$\infer{\nomi_0\leq\phi_i\ \ \ \psi_i\leq \neg\nomi_1}{\phi_i\leq\psi_i}
$$

Here, $\nomi_0$ and $\nomi_1$ are special fresh nominals. Now we get a set of inequalities $\{\nomi_0\leq\phi_i, \psi_i\leq \neg\nomi_1\}_{i\in I}$.

\item \textbf{The reduction stage}:

In this stage, for each $\{\nomi_0\leq\phi_i, \psi_i\leq \neg\nomi_1\}$, we first add superscripts and subscripts $\emptyset$ to the two $\leq$s, and then apply the following rules to prepare for eliminating all the proposition variables in $\{\nomi_0\leq^{\emptyset}_{\emptyset}\phi_i, \psi_i\leq^{\emptyset}_{\emptyset} \neg\nomi_1\}$:

\begin{enumerate}

\item \textbf{Substage 1: Decomposing the outer part}

In the current substage, the following rules are applied to decompose the outer part of the Sahlqvist signed formula:

\begin{enumerate}

\item Splitting rules:

$$\infer{\alpha\leq^{S}_{S'}\beta\ \ \ \alpha\leq^{S}_{S'}\gamma}{\alpha\leq^{S}_{S'}\beta\land\gamma}
\qquad
\infer{\alpha\leq^{S}_{S'}\gamma\ \ \ \beta\leq^{S}_{S'}\gamma}{\alpha\lor\beta\leq^{S}_{S'}\gamma}
$$

\item Approximation rules:
$$\infer{\nomj\leq^{S'}_{S'}\alpha\ \ \ \nomi\leq^{S}_{S'}\Diamond^{S'} \nomj}{\nomi\leq^{S}_{S'}\Diamond\alpha}
\qquad
\infer{\alpha\leq^{S}_{S} \neg\nomj\ \ \ \Box^{S} \neg\nomj\leq^{S}_{S'}\neg \nomi}{\Box\alpha\leq^{S}_{S'}\neg \nomi}
$$

$$\infer{\nomi_{m0}\leq^{S'}_{S'}\Diamond^{S'}\nomi_{m1}\ \ \ \ \ \ \ \nomi \leq^{S}_{S'\cup\{(\nomi_{m0},\nomi_{m1})\}}\alpha}{\nomi\leq^{S}_{S'}\Diamondblack\alpha}
\qquad
\infer{\nomi_{m0}\leq^{S}_{S}\Diamond^{S}\nomi_{m1}\ \ \ \ \ \ \alpha\leq^{S\cup\{(\nomi_{m0},\nomi_{m1})\}}_{S'}\neg\nomi}{\blacksquare\alpha\leq^{S}_{S'} \neg\nomi}
$$

$$\infer{\nomj\leq^{S}_{S}\alpha\ \ \ \ \ \ \ \beta\leq^{S}_{S}\neg\nomk\ \ \ \ \ \ \ \nomj\rightarrow\neg\nomk\leq^{S}_{S'}\neg\nomi}{\alpha\rightarrow\beta\leq^{S}_{S'}\neg\nomi}
$$

The nominals introduced by the approximation rules must not occur in the system before applying the rule.

\item Residuation rules:

$$\infer{\alpha\leq^{S'}_{S}\neg\nomi}{\nomi\leq^{S}_{S'}\neg\alpha}
\qquad
\infer{\nomi\leq^{S'}_{S}\alpha}{\neg\alpha\leq^{S}_{S'}\neg\nomi}
$$

\end{enumerate}

\item \textbf{Substage 2: Decomposing the inner part}

In the current substage, the following rules are applied to decompose the inner part of the Sahlqvist signed formula:

\begin{enumerate}

\item Splitting rules:

$$\infer{\alpha\leq^{S}_{S'}\beta\ \ \ \alpha\leq^{S}_{S'}\gamma}{\alpha\leq^{S}_{S'}\beta\land\gamma}
\qquad
\infer{\alpha\leq^{S}_{S'}\gamma\ \ \ \beta\leq^{S}_{S'}\gamma}{\alpha\lor\beta\leq^{S}_{S'}\gamma}
$$

\item Residuation rules:

$$\infer{\beta\leq^{S'}_{S}\neg\alpha}{\alpha\leq^{S}_{S'}\neg\beta}
\qquad
\infer{\neg\beta\leq^{S'}_{S}\alpha}{\neg\alpha\leq^{S}_{S'}\beta}
\qquad
\infer{\alpha\leq^{S}_{S'}(\Box^{S})^{-1}\beta}{\Diamond\alpha\leq^{S}_{S'}\beta}
\qquad
\infer{(\Diamond^{S'})^{-1}\alpha\leq^{S}_{S'}\beta}{\alpha\leq^{S}_{S'}\Box\beta}
$$

$$\infer{\forall\nomi_{m0}\forall\nomi_{m1}(\nomi_{m0}\leq^{S'}_{S'}\Diamond^{S'}\nomi_{m1}\ \Rightarrow\ \alpha \leq^{S}_{S'\cup\{(\nomi_{m0},\nomi_{m1})\}} \beta)}{\alpha\leq^{S}_{S'}\blacksquare\beta}
$$

$$\infer{\forall\nomi_{m0}\forall\nomi_{m1}(\nomi_{m0}\leq^{S}_{S}\Diamond^{S}\nomi_{m1}\ \Rightarrow\ \alpha \leq^{S\cup\{(\nomi_{m0},\nomi_{m1})\}}_{S'} \beta)}{\Diamondblack\alpha\leq^{S}_{S'}\beta}
$$

The nominals introduced by the residuation rules must not occur in the system before applying the rule.

\item Second splitting rules:
$$\infer{\forall\nomi_{m0}\forall\nomi_{m1}
(\nomi_{m0}\leq^{S}_{S}\Diamond^{S}\nomi_{m1}\Rightarrow\mathsf{Mega_1})
\ \ \ 
\forall\nomi_{m0}\forall\nomi_{m1}
(\nomi_{m0}\leq^{S}_{S}\Diamond^{S}\nomi_{m1}\Rightarrow\mathsf{Mega_2})}
{\forall\nomi_{m0}\forall\nomi_{m1}
(\nomi_{m0}\leq^{S}_{S}\Diamond^{S}\nomi_{m1} \Rightarrow\ \mathsf{Mega_1} \bigamp\mathsf{Mega_2})}$$

Here $\mathsf{Mega_1}$ and $\mathsf{Mega_2}$ denote mega-inequalities.

\end{enumerate}

\item \textbf{Substage 3: Preparing for the Ackermann rules}

In this substage, we prepare for eliminating the propositional variables by the Ackermann rules, with the help of the following packing rules:\\

Packing rules:
$$\infer{\exists\nomi_{m_{0}0}\exists\nomi_{m_{0}1}\ldots\exists\nomi_{m_{k}0}\exists\nomi_{m_{k}1}(
A(\nomi_{m_{0}0}\to\Diamond^{S_{0}}\nomi_{m_{0}1})
\land
A(\nomi_{m_{k}0}\to\Diamond^{S_{k}}\nomi_{m_{k}1})
\land\alpha)\leq^{\emptyset}_{\emptyset} p}
{\forall\nomi_{m_{0}0}\forall\nomi_{m_{0}1}
(\nomi_{m_{0}0}\leq^{S_{0}}_{S_{0}}\Diamond^{S_{0}}\nomi_{m_{0}1}\Rightarrow\ldots
\forall\nomi_{m_{k}0}\forall\nomi_{m_{k}1}(\nomi_{m_{k}0}\leq^{S_{k}}_{S_{k}}\Diamond^{S_{k}}\nomi_{m_{k}1}\ \Rightarrow\ \alpha\leq^{S}_{S'} p)\ldots)}$$
where $\alpha$ is pure and does not contain contextual connectives $\Box,\Diamond,\blacksquare,\Diamondblack$;

$$\infer{p\leq^{\emptyset}_{\emptyset}
\forall\nomi_{m_{0}0}\forall\nomi_{m_{0}1}\ldots\forall\nomi_{m_{k}0}\forall\nomi_{m_{k}1}(A(\nomi_{m_{0}0}\to\Diamond^{S_{0}}\nomi_{m_{0}1})
\land
A(\nomi_{m_{k}0}\to\Diamond^{S_{k}}\nomi_{m_{k}1})
\to\beta)}
{\forall\nomi_{m_{0}0}\forall\nomi_{m_{0}1}
(\nomi_{m_{0}0}\leq^{S_{0}}_{S_{0}}\Diamond^{S_{0}}\nomi_{m_{0}1}\Rightarrow\ldots
\forall\nomi_{m_{k}0}\forall\nomi_{m_{k}1}(\nomi_{m_{k}0}\leq^{S_{k}}_{S_{k}}\Diamond^{S_{k}}\nomi_{m_{k}1}\ \Rightarrow\ p\leq^{S}_{S'}\beta)\ldots)}$$
where $\beta$ is pure and does not contain contextual connectives $\Box,\Diamond,\blacksquare,\Diamondblack$.

$$\infer{\forall\nomi_{m_{0}0}\forall\nomi_{m_{0}1}\ldots\forall\nomi_{m_{k}0}\forall\nomi_{m_{k}1}(\top\leq^{\emptyset}_{S'}
A(\nomi_{m_{0}0}\to\Diamond^{S_{0}}\nomi_{m_{0}1})
\land
A(\nomi_{m_{k}0}\to\Diamond^{S_{k}}\nomi_{m_{k}1})\land\alpha
\to\gamma)}
{\forall\nomi_{m_{0}0}\forall\nomi_{m_{0}1}
(\nomi_{m_{0}0}\leq^{S_{0}}_{S_{0}}\Diamond^{S_{0}}\nomi_{m_{0}1}\Rightarrow\ldots
\forall\nomi_{m_{k}0}\forall\nomi_{m_{k}1}(\nomi_{m_{k}0}\leq^{S_{k}}_{S_{k}}\Diamond^{S_{k}}\nomi_{m_{k}1}\ \Rightarrow\ \alpha\leq^{S}_{S'}\gamma)\ldots)}$$
where $\alpha$ is pure and does not contain contextual connectives $\Box,\Diamond,\blacksquare,\Diamondblack$.

$$\infer{\forall\nomi_{m_{0}0}\forall\nomi_{m_{0}1}\ldots\forall\nomi_{m_{k}0}\forall\nomi_{m_{k}1}(\top\leq^{\emptyset}_{S}
A(\nomi_{m_{0}0}\to\Diamond^{S_{0}}\nomi_{m_{0}1})
\land
A(\nomi_{m_{k}0}\to\Diamond^{S_{k}}\nomi_{m_{k}1})\land\gamma
\to\alpha)}
{\forall\nomi_{m_{0}0}\forall\nomi_{m_{0}1}
(\nomi_{m_{0}0}\leq^{S_{0}}_{S_{0}}\Diamond^{S_{0}}\nomi_{m_{0}1}\Rightarrow\ldots
\forall\nomi_{m_{k}0}\forall\nomi_{m_{k}1}(\nomi_{m_{k}0}\leq^{S_{k}}_{S_{k}}\Diamond^{S_{k}}\nomi_{m_{k}1}\ \Rightarrow\ \gamma\leq^{S}_{S'}\alpha)\ldots)}$$
where $\alpha$ is pure and does not contain contextual connectives $\Box,\Diamond,\blacksquare,\Diamondblack$.

\item \textbf{Substage 4: The Ackermann stage}

In this substage, we compute the minimal/maximal valuation for propositional variables and use the Ackermann rules to eliminate all the propositional variables.  These two rules are the core of $\mathsf{ALBA}$, since their application eliminates proposition variables. In fact, all the preceding steps are aimed at reaching a shape in which the rules can be applied. Notice that an important feature of these rules is that they are executed on the whole set of (universally quantified) inequalities, and not on a single inequality.\\

The right-handed Ackermann rule:

The system 
$\left\{ \begin{array}{ll}
\alpha_1\leq^{\emptyset}_{\emptyset} p \\
\vdots\\
\alpha_n\leq^{\emptyset}_{\emptyset} p \\
\forall\vec{\nomi_1}(\beta_1\leq^{T_1}_{T'_1}\gamma_1)\\
\vdots\\
\forall\vec{\nomi_m}(\beta_m\leq^{T_m}_{T'_m}\gamma_m)\\

\end{array} \right.$ 

is replaced by 
$\left\{ \begin{array}{ll}
\forall\vec{\nomi_1}(\beta_1((\alpha_1\lor\ldots\lor\alpha_n)/p)\leq^{T_1}_{T'_1}\gamma_1((\alpha_1\lor\ldots\lor\alpha_n)/p)) \\
\vdots\\
\forall\vec{\nomi_m}(\beta_m((\alpha_1\lor\ldots\lor\alpha_n)/p)\leq^{T_m}_{T'_m}\gamma_m((\alpha_1\lor\ldots\lor\alpha_n)/p)) \\

\end{array} \right.$

where:

\begin{enumerate}
\item $p, \vec{\nomi_1}, \ldots, \vec{\nomi_m}$ do not occur in $\alpha_1, \ldots, \alpha_n$;
\item Each $\beta_i$ is positive in $p$, and each $\gamma_i$ negative in $p$, for $1\leq i\leq m$;
\item Each $\alpha_i$ is pure and contains no contextual modalities $\Box, \Diamond, \blacksquare, \Diamondblack$.
\end{enumerate}

The left-handed Ackermann rule:

The system
$\left\{ \begin{array}{ll}
p\leq^{\emptyset}_{\emptyset}\alpha_1 \\
\vdots\\
p\leq^{\emptyset}_{\emptyset}\alpha_n \\
\forall\vec{\nomi_1}(\beta_1\leq^{T_1}_{T'_1}\gamma_1)\\
\vdots\\
\forall\vec{\nomi_m}(\beta_m\leq^{T_m}_{T'_m}\gamma_m)\\

\end{array} \right.$

is replaced by
$\left\{ \begin{array}{ll}
\forall\vec{\nomi_1}(\beta_1((\alpha_1\land\ldots\land\alpha_n)/p)\leq^{T_1}_{T'_1}\gamma_1((\alpha_1\land\ldots\land\alpha_n)/p)) \\
\vdots\\
\forall\vec{\nomi_m}(\beta_m((\alpha_1\land\ldots\land\alpha_n)/p)\leq^{T_m}_{T'_m}\gamma_m((\alpha_1\land\ldots\land\alpha_n)/p)) \\

\end{array} \right.$

where:
\begin{enumerate}
\item $p, \vec{\nomi_1}, \ldots, \vec{\nomi_m}$ do not occur in $\alpha_1, \ldots, \alpha_n$;
\item Each $\beta_i$ is negative in $p$, and each $\gamma_i$ positive in $p$, for $1\leq i\leq m$.
\item Each $\alpha_i$ is pure and contains no contextual modalities $\Box, \Diamond, \blacksquare, \Diamondblack$.
\end{enumerate}
\end{enumerate}

\item \textbf{Output}: If in the previous stage, for some $\{\nomi_0\leq\phi_i, \psi_i\leq \neg\nomi_1\}$, the algorithm gets stuck, i.e.\ some proposition variables cannot be eliminated by the application of the reduction rules, then the algorithm halts and output ``failure''. Otherwise, each initial tuple $\{\nomi_0\leq\phi_i, \psi_i\leq \neg\nomi_1\}$ of inequalities after the first approximation has been reduced to a set of pure (universally quantified) inequalities $\mathsf{Reduce}(\phi_i\leq\psi_i)$, and then the output is a set of quasi-(universally quantified) inequalities $\{\&\mathsf{Reduce}(\phi_i\leq\psi_i)\Rightarrow \nomi_0\leq \neg\nomi_1: \phi_i\leq\psi_i\in\mathsf{Preprocess}(\phi\leq\psi)\}$, where $\&$ is the big meta-conjunction in quasi-inequalities. Then the algorithm use the standard translation to transform the quasi-(universally quantified) inequalities into first-order formulas.
\end{enumerate}

\subsection{Soundness of $\mathsf{ALBA}^{\mathsf{SML}}$}\label{Subsec:soundness}

In the present subsection, we will prove the soundness of the algorithm $\mathsf{ALBA}^{\mathsf{SML}}$ with respect to Kripke frames. The basic proof structure is similar to \cite{CoPa12}.

\begin{theorem}[Soundness]\label{Thm:Soundness}
If $\mathsf{ALBA}^{\mathsf{SML}}$ runs successfully on $\phi\leq\psi$ and outputs $\mathsf{FO}(\phi\leq\psi)$, then for any Kripke frame $\mathbb{F}=(W,R_0)$, $$\mathbb{F}\Vdash\phi \leq \psi\mbox{ iff }\mathbb{F}\models\mathsf{FO}(\phi \leq \psi).$$
\end{theorem}

\begin{proof}
The proof goes similarly to \cite[Theorem 8.1]{CoPa12}. Let $\phi_i\leq\psi_i$, $1\leq i\leq n$ denote the inequalities produced by preprocessing $\phi\leq\psi$ after Stage 1, and $\{\nomi_{0}\leq\phi_{i}, \psi_{i}\leq\neg\nomi_{1}\}$ denote the inequalities after the first-approximation rule, $\mathsf{Reduce}(\phi_i\leq\psi_i)$ denote the set of pure (universally quantified) inequalities after Stage 2, and $\mathsf{FO}(\phi \leq \psi)$ denote the standard translation of the quasi-(universally quantified) inequalities into first-order formulas, then we have the following chain of equivalences:

It suffices to show the equivalence from (\ref{aCrct:Eqn0}) to (\ref{aCrct:Eqn4}) given below:

\begin{eqnarray}
&&\mathbb{F}\Vdash\phi\leq\psi\label{aCrct:Eqn0}\\
&&\mathbb{F}\Vdash\phi_i\leq\psi_i,\mbox{ for all }1\leq i\leq n\label{aCrct:Eqn1}\\
&&\mathbb{F}\Vdash(\nomi_{0}\leq^{\emptyset}_{\emptyset}\phi_{i}\ \&\ \psi_{i}\leq^{\emptyset}_{\emptyset}\neg\nomi_{1})\Rightarrow\nomi_{0}\leq\neg\nomi_{1} \mbox{ for all }1\leq i\leq n\label{aCrct:Eqn2}\\
&&\mathbb{F}\Vdash\mathsf{Reduce}(\phi_{i}\leq\psi_{i})\Rightarrow\nomi_{0}\leq\neg\nomi_{1} \mbox{ for all }1\leq i\leq n\label{aCrct:Eqn3}\\
&&\mathbb{F}\Vdash\mathsf{FO}(\phi\leq\psi)\label{aCrct:Eqn4}
\end{eqnarray}

\begin{itemize}
\item The equivalence between (\ref{aCrct:Eqn0}) and (\ref{aCrct:Eqn1}) follows from Proposition \ref{prop:Soundness:stage:1};
\item the equivalence between (\ref{aCrct:Eqn1}) and (\ref{aCrct:Eqn2}) follows from Proposition \ref{prop:Soundness:first:approximation};
\item the equivalence between (\ref{aCrct:Eqn2}) and (\ref{aCrct:Eqn3}) follows from Propositions \ref{Prop:Substage:1}, \ref{Prop:Substage:2}, \ref{Prop:Substage:3}, \ref{Prop:Substage:4};
\item the equivalence between (\ref{aCrct:Eqn3}) and (\ref{aCrct:Eqn4}) follows from Proposition \ref{Prop:ST:ineq:quasi:mega}.
\end{itemize}
\end{proof}

In the remainder of this subsection, we prove the soundness of the rules in Stage 1, 2 and 3.

\begin{proposition}[Soundness of the rules in Stage 1]\label{prop:Soundness:stage:1}
For the distribution rules, the splitting rules and the monotone and antitone variable-elimination rules, they are sound in both directions in $\mathbb{F}$, i.e.,\ it is sound from the premise to the conclusion and the other way round.
\end{proposition}

\begin{proof}
For the soundness of the distribution rules, it follows from the fact that the following equivalences are valid in $\mathbb{F}$:
\begin{itemize}
\item $\Diamond(\alpha\lor\beta)\leftrightarrow\Diamond\alpha\lor\Diamond\beta$;
\item $\Diamondblack(\alpha\lor\beta)\leftrightarrow\Diamondblack\alpha\lor\Diamondblack\beta$;
\item $\neg(\alpha\lor\beta)\leftrightarrow\neg\alpha\land\neg\beta$;
\item $(\alpha\lor\beta)\land\gamma\leftrightarrow(\alpha\land\gamma)\lor(\beta\land\gamma)$;
\item $\alpha\land(\beta\lor\gamma)\leftrightarrow(\alpha\land\beta)\lor(\alpha\land\gamma)$;
\item $((\alpha\lor\beta)\to\gamma)\leftrightarrow((\alpha\to\gamma)\land(\beta\to\gamma))$;
\item $\Box(\alpha\land\beta)\leftrightarrow\Box\alpha\land\Box\beta$;
\item $\blacksquare(\alpha\land\beta)\leftrightarrow\blacksquare\alpha\land\blacksquare\beta$;
\item $\neg(\alpha\land\beta)\leftrightarrow\neg\alpha\lor\neg\beta$;
\item $(\alpha\land\beta)\lor\gamma\leftrightarrow(\alpha\lor\gamma)\land(\beta\lor\gamma)$;
\item $\alpha\lor(\beta\land\gamma)\leftrightarrow(\alpha\lor\beta)\land(\alpha\lor\gamma)$;
\item $(\alpha\to\beta\land\gamma)\leftrightarrow(\alpha\to\beta)\land(\alpha\to\gamma)$.
\end{itemize}

For the soundness of the splitting rules, it follows from the following fact:

$$\mathbb{F}\Vdash\alpha\leq\beta\land\gamma\mbox{ iff }(\mathbb{F}\Vdash\alpha\leq\beta\mbox{ and }\mathbb{F}\Vdash\alpha\leq\gamma);$$
$$\mathbb{F}\Vdash\alpha\lor\beta\leq\gamma\mbox{ iff }(\mathbb{F}\Vdash\alpha\leq\gamma\mbox{ and }\mathbb{F}\Vdash\beta\leq\gamma).$$

For the soundness of the monotone and antitone variable elimination rules, we show the soundness for the first rule. Suppose $\alpha(p)$ is negative in $p$ and $\beta$ is positive in $p$. 

If $\mathbb{F}\Vdash\alpha(p)\leq\beta(p)$, then for all valuations $V$, $(\mathbb{F}, V)\Vdash\alpha(p)\leq\beta(p)$, thus for the valuation $V^{p}_{\emptyset}$ such that $V^{p}_{\emptyset}$ is the same as $V$ except that $V^{p}_{\emptyset}(p)=\emptyset$, $(\mathbb{F}, V^{p}_{\emptyset})\Vdash\alpha(p)\leq\beta(p)$, therefore $(\mathbb{F}, V^{p}_{\emptyset})\Vdash\alpha(\bot)\leq\beta(\bot)$, thus $(\mathbb{F}, V)\Vdash\alpha(\bot)\leq\beta(\bot)$, so $\mathbb{F}\Vdash\alpha(\bot)\leq\beta(\bot)$.

For the other direction, suppose $\mathbb{F}\vDash\alpha(\bot)\leq\beta(\bot)$, then by the fact that $\alpha(p)$ is negative in $p$ and $\beta$ is positive in $p$, we have that $\mathbb{F}\vDash\alpha(p)\leq\alpha(\bot)$ and $\mathbb{F}\vDash\beta(\bot)\leq\beta(p)$, therefore $\mathbb{F}\vDash\alpha(p)\leq\beta(p)$.

The soundness of the other rule is similar.
\end{proof}

\begin{proposition}\label{prop:Soundness:first:approximation}
(\ref{aCrct:Eqn1}) and (\ref{aCrct:Eqn2}) are equivalent, i.e.\ the first-approximation rule is sound in $\mathbb{F}$.
\end{proposition}

\begin{proof}
(\ref{aCrct:Eqn1}) $\Rightarrow$ (\ref{aCrct:Eqn2}): Suppose $\mathbb{F}\Vdash\phi_i\leq\psi_i$. Then for any valuation $V$ on $\mathbb{F}$, if $(\mathbb{F},V)\Vdash\nomi_{0}\leq^{\emptyset}_{\emptyset}\phi_{i}$ and $(\mathbb{F},V)\Vdash\psi_{i}\leq^{\emptyset}_{\emptyset}\neg\nomi_{1}$, then $(\mathbb{F},V),V(\nomi_{0})\Vdash\phi_{i}$ and $(\mathbb{F},V),V(\nomi_{1})\nVdash\psi_{i}$, so by $\mathbb{F}\Vdash\phi_{i}\leq\psi_{i}$ we have $(\mathbb{F},V),V(\nomi_{0})\Vdash\psi_{i}$, so $\nomi_{0}\neq\nomi_{1}$, so $(\mathbb{F},V)\Vdash\nomi_{0}\leq\neg\nomi_{1}$.

(\ref{aCrct:Eqn2}) $\Rightarrow$ (\ref{aCrct:Eqn1}): Suppose $\mathbb{F}\Vdash(\nomi_{0}\leq^{\emptyset}_{\emptyset}\phi_{i}\ \&\ \psi_{i}\leq^{\emptyset}_{\emptyset}\neg\nomi_{1})\Rightarrow\nomi_{0}\leq\neg\nomi_{1}$. Then if $\mathbb{F}\nVdash\phi_i\leq\psi_i$, then there is a valuation $V$ on $\mathbb{F}$ and a $w\in W$ such that $(\mathbb{F}, V), w\Vdash\phi_i$ and $(\mathbb{F}, V), w\nVdash\psi_i$. Then by taking $V^{\nomi_{0}, \nomi_{1}}_{w, w}$ to be the valuation which is the same as $V$ except that $V^{\nomi_{0}, \nomi_{1}}_{w, w}(\nomi_{0})=V^{\nomi_{0}, \nomi_{1}}_{w, w}(\nomi_{1})=\{w\}$, then since $\nomi_{0}, \nomi_{1}$ do not occur in $\phi_{i}$ and $\psi_{i}$, we have that $(\mathbb{F}, V^{\nomi_{0}, \nomi_{1}}_{w, w}), w\Vdash\phi_i$ and $(\mathbb{F}, V^{\nomi_{0}, \nomi_{1}}_{w, w}), w\nVdash\psi_i$, therefore $(\mathbb{F}, V^{\nomi_{0}, \nomi_{1}}_{w, w})\Vdash\nomi_{0}\leq^{\emptyset}_{\emptyset}\phi_i$ and $(\mathbb{F}, V^{\nomi_{0}, \nomi_{1}}_{w, w})\Vdash\psi_i\leq^{\emptyset}_{\emptyset}\neg\nomi_{1}$, by $\mathbb{F}\Vdash(\nomi_{0}\leq^{\emptyset}_{\emptyset}\phi_{i}\ \&\ \psi_{i}\leq^{\emptyset}_{\emptyset}\neg\nomi_{1})\Rightarrow\nomi_{0}\leq\neg\nomi_{1}$, we have that $(\mathbb{F}, V^{\nomi_{0}, \nomi_{1}}_{w, w})\Vdash\nomi_{0}\leq\neg\nomi_{1}$, so $(\mathbb{F}, V^{\nomi_{0}, \nomi_{1}}_{w, w}),w\Vdash\nomi_{0}$ implies that $(\mathbb{F}, V^{\nomi_{0}, \nomi_{1}}_{w, w}),w\Vdash\neg\nomi_{1}$, a contradiction. So $\mathbb{F}\Vdash\phi_i\leq\psi_i$.
\end{proof}

The next step is to show the soundness of Stage 2, for which it suffices to show the soundness of each rule in each substage.

\begin{proposition}\label{Prop:Substage:1}
The splitting rules, the approximation rules for $\Diamond,\Box,\Diamondblack,\blacksquare,\to$, the residuation rules for $\neg$ in Substage 1 are sound in $\mathbb{F}$.
\end{proposition}

\begin{proof}
By Lemma \ref{Lemma:Splitting:substage12}, \ref{Lemma:approximation:white:substage1}, \ref{Lemma:approximation:black:Substage1}, \ref{Lemma:Approximation:to:substage1} and \ref{Lemma:residuation:neg:substage12} below.
\end{proof}

\begin{lemma}\label{Lemma:Splitting:substage12}
The splitting rules in Substage 1 and Substage 2 are sound in $\mathbb{F}$.
\end{lemma}

\begin{proof}
For the soundness of the splitting rules, it follows from the fact that for any Kripke frame $\mathbb{F}=(W,R_0)$, any valuation $V$ on $\mathbb{F}$, 
\begin{itemize}
\item $(\mathbb{F},V)\Vdash\alpha\leq^{S}_{S'}\beta\land\gamma\mbox{ iff }((\mathbb{F},V)\Vdash\alpha\leq^{S}_{S'}\beta\mbox{ and }(\mathbb{F},V)\Vdash\alpha\leq^{S}_{S'}\gamma),$
\item $(\mathbb{F},V)\Vdash\alpha\lor\beta\leq^{S}_{S'}\gamma\mbox{ iff }((\mathbb{F},V)\Vdash\alpha\leq^{S}_{S'}\gamma\mbox{ and }(\mathbb{F},V)\Vdash\beta\leq^{S}_{S'}\gamma).$
\end{itemize}
\end{proof}

\begin{lemma}\label{Lemma:approximation:white:substage1}
The approximation rules for $\Diamond,\Box$ in Substage 1 are sound in $\mathbb{F}$.
\end{lemma}

\begin{proof}
We prove for $\Diamond$, the case for $\Box$ is similar.
For the soundness of the approximation rule for $\Diamond$, it suffices to show that for any Kripke frame $\mathbb{F}=(W,R_0)$, any valuation $V$ on $\mathbb{F}$, 
\begin{enumerate}
\item if $(\mathbb{F},V)\Vdash\nomi\leq^{S}_{S'}\Diamond\alpha$, then there is a valuation $V^{\nomj}$ such that $V^{\nomj}$ is the same as $V$ except $V^{\nomj}(\nomj)$, and $(\mathbb{F},V^{\nomj})\Vdash\nomi\leq^{S}_{S'}\Diamond^{S'}\nomj$ and $(\mathbb{F},V^{\nomj})\Vdash\nomj\leq^{S'}_{S'}\alpha$;
\item if $(\mathbb{F},V)\Vdash\nomi\leq^{S}_{S'}\Diamond^{S'}\nomj$ and $(\mathbb{F},V)\Vdash\nomj\leq^{S'}_{S'}\alpha$, then $(\mathbb{F},V)\Vdash\nomi\leq^{S}_{S'}\Diamond\alpha$.
\end{enumerate}
For item 1, if $(\mathbb{F},V)\Vdash\nomi\leq^{S}_{S'}\Diamond\alpha$, then $(W,(R_0\setminus S'), V), V(\nomi)\Vdash\Diamond\alpha$, therefore there exists a $w\in W$ such that $(V(\nomi),w)\in (R_0\setminus S')$ and $(W,(R_0\setminus S'), V), w\Vdash\alpha$. Now take $V^{\nomj}$ such that $V^{\nomj}$ is the same as $V$ except that $V^{\nomj}(\nomj)=\{w\}$, then $(V^{\nomj}(\nomi), V^{\nomj}(\nomj))\in (R_0\setminus S')$, so $(\mathbb{F},V^{\nomj})\Vdash\nomi\leq^{S}_{S'}\Diamond^{S'}\nomj$ and $(\mathbb{F},V^{\nomj})\Vdash\nomj\leq^{S'}_{S'}\alpha$.

For item 2, suppose $(\mathbb{F},V)\Vdash\nomi\leq^{S}_{S'}\Diamond^{S'}\nomj$ and $(\mathbb{F},V)\Vdash\nomj\leq^{S'}_{S'}\alpha$. Then $(V(\nomi), V(\nomj))\in (R_0\setminus S')$ and $(W,(R_0\setminus S'),V),V(\nomj)\Vdash\alpha$, so $(W,(R_0\setminus S'),V),V(\nomi)\Vdash\Diamond\alpha$, therefore $(\mathbb{F},V)\Vdash\nomi\leq^{S}_{S'}\Diamond\alpha$.
\end{proof}

\begin{lemma}\label{Lemma:approximation:black:Substage1}
The approximation rules for $\Diamondblack,\blacksquare$ in Substage 1 are sound in $\mathbb{F}$.
\end{lemma}

\begin{proof}
We prove for $\Diamondblack$, the case for $\blacksquare$ is similar. For the soundness of the approximation rule for $\Diamondblack$, it suffices to show that for any Kripke frame $\mathbb{F}=(W,R_0)$, any valuation $V$ on $\mathbb{F}$, 
\begin{enumerate}
\item if $(\mathbb{F}, V)\Vdash\nomi\leq^{S}_{S'}\Diamondblack\alpha$, then there is a valuation $V^{\nomi_{m0}, \nomi_{m1}}$ such that $V^{\nomi_{m0}, \nomi_{m1}}$ is the same as $V$ except $V^{\nomi_{m0}, \nomi_{m1}}(\nomi_{m0})$ and $V^{\nomi_{m0}, \nomi_{m1}}(\nomi_{m1})$, and $(\mathbb{F}, V^{\nomi_{m0}, \nomi_{m1}})\Vdash\nomi_{m0}\leq^{S'}_{S'}\Diamond^{S'}\nomi_{m1}$ and $(\mathbb{F}, V^{\nomi_{m0}, \nomi_{m1}})\Vdash\nomi\leq^{S}_{S'\cup\{(\nomi_{m0},\nomi_{m1})\}}\alpha$;
\item if $(\mathbb{F}, V)\Vdash\nomi_{m0}\leq^{S'}_{S'}\Diamond^{S'}\nomi_{m1}$ and $(\mathbb{F}, V)\Vdash\nomi\leq^{S}_{S'\cup\{(\nomi_{m0},\nomi_{m1})\}}\alpha$, then $(\mathbb{F}, V)\Vdash\nomi\leq^{S}_{S'}\Diamondblack\alpha$.
\end{enumerate}
For item 1, if $(\mathbb{F}, V)\Vdash\nomi\leq^{S}_{S'}\Diamondblack\alpha$, then $(W,(R_0\setminus S'), V), V(\nomi)\Vdash\Diamondblack\alpha$, therefore there are $(w, v)\in(R_0\setminus S')$ such that $(W, ((R_0\setminus S')\setminus\{(w,v)\}), V), V(\nomi)\Vdash\alpha$. Now take $V^{\nomi_{m0}, \nomi_{m1}}$ such that $V^{\nomi_{m0}, \nomi_{m1}}$ is the same as $V$ except $V^{\nomi_{m0}, \nomi_{m1}}(\nomi_{m0})=\{w\}$ and $V^{\nomi_{m0}, \nomi_{m1}}(\nomi_{m1})=\{v\}$, then by $(w,v)\in (R_0\setminus S')$, we have that $(W, R_0, V^{\nomi_{m0}, \nomi_{m1}})\Vdash \nomi_{m0}\leq^{S'}_{S'}\Diamond^{S'}\nomi_{m1}$, and from $(W, ((R_0\setminus S')\setminus\{(w,v)\}), V), V(\nomi)\Vdash\alpha$ we have that $(W, ((R_0\setminus S')\setminus\{(V^{\nomi_{m0}, \nomi_{m1}}(\nomi_{m0}),V^{\nomi_{m0}, \nomi_{m1}}(\nomi_{m1}))\}),V^{\nomi_{m0}, \nomi_{m1}}), V^{\nomi_{m0}, \nomi_{m1}}(\nomi)\Vdash\alpha$, so $(\mathbb{F}, V^{\nomi_{m0}, \nomi_{m1}})\Vdash\nomi\leq^{S}_{S'\cup\{(\nomi_{m0},\nomi_{m1})\}}\alpha$.

For item 2, if $(\mathbb{F}, V)\Vdash\nomi_{m0}\leq^{S'}_{S'}\Diamond^{S'}\nomi_{m1}$ and $(\mathbb{F},V)\Vdash\nomi\leq^{S}_{S'\cup\{(\nomi_{m0},\nomi_{m1})\}}\alpha$, then $(V(\nomi_{m0}), V(\nomi_{m1}))\in (R_0\setminus S')$, and $(W,(R_0\setminus (S'\cup\{(\nomi_{m0},\nomi_{m1})\})),V),V(\nomi)\Vdash\alpha$, so $(W, (R_0\setminus S'), V), V(\nomi)\Vdash\Diamondblack\alpha$, therefore $(\mathbb{F}, V)\Vdash\nomi\leq^{S}_{S'}\Diamondblack\alpha$.
\end{proof}

\begin{lemma}\label{Lemma:Approximation:to:substage1}
The approximation rule for $\to$ in Substage 1 is sound in $\mathbb{F}$.
\end{lemma}

\begin{proof}
For the soundness of the approximation rule for $\to$, it suffices to show that for any Kripke frame $\mathbb{F}=(W,R_0)$, any valuation $V$ on $\mathbb{F}$, 
\begin{enumerate}
\item if $(\mathbb{F},V)\Vdash\alpha\to\beta\leq^{S}_{S'}\neg\nomi$, then there is a valuation $V^{\nomj,\nomk}$ such that $V^{\nomj,\nomk}$ is the same as $V$ except $V^{\nomj,\nomk}(\nomj)$ and $V^{\nomj,\nomk}(\nomk)$, and $(\mathbb{F},V^{\nomj,\nomk})\Vdash\nomj\leq^{S}_{S}\alpha$, $(\mathbb{F},V^{\nomj,\nomk})\Vdash\beta\leq^{S}_{S}\neg \nomk$ and $(\mathbb{F},V^{\nomj,\nomk})\Vdash\nomj\to\neg\nomk\leq^{S}_{S'}\neg\nomi$; 
\item if $(\mathbb{F},V)\Vdash\nomj\leq^{S}_{S}\alpha$, $(\mathbb{F},V)\Vdash\beta\leq^{S}_{S}\neg \nomk$ and $(\mathbb{F},V)\Vdash\nomj\to\neg\nomk\leq^{S}_{S'}\neg\nomi$, then $(\mathbb{F},V)\Vdash\alpha\to\beta\leq^{S}_{S'}\neg\nomi$.

For item 1, if $(\mathbb{F},V)\Vdash\alpha\to\beta\leq^{S}_{S'}\neg\nomi$, then $(W, (R_0\setminus S),V), V(\nomi)\Vdash\alpha$ and $(W, (R_0\setminus S),V), V(\nomi)\Vdash\neg\beta$. Now take $V^{\nomj,\nomk}$ such that $V^{\nomj,\nomk}$ is the same as $V$ except that $V^{\nomj,\nomk}(\nomj)=V^{\nomj,\nomk}(\nomk)=V(\nomi)$, we have that $(W, (R_0\setminus S),V^{\nomj,\nomk}), V^{\nomj,\nomk}(\nomj)\Vdash\alpha$ and $(W, (R_0\setminus S),V^{\nomj,\nomk}), V^{\nomj,\nomk}(\nomk)\Vdash\neg\beta$, so $(\mathbb{F},V^{\nomj,\nomk})\Vdash\nomj\leq^{S}_{S}\alpha$, $(\mathbb{F},V^{\nomj,\nomk})\Vdash\beta\leq^{S}_{S}\neg \nomk$. Since $V^{\nomj,\nomk}(\nomj)=V^{\nomj,\nomk}(\nomk)=V^{\nomj,\nomk}(\nomi)=V(\nomi)$, it is easy to see that $V^{\nomj,\nomk}(\nomj\to\neg\nomk)=V^{\nomj,\nomk}(\neg\nomi)$, so $(\mathbb{F},V^{\nomj,\nomk})\Vdash\nomj\to\neg\nomk\leq^{S}_{S'}\neg\nomi$.

For item 2, if $(\mathbb{F},V)\Vdash\nomj\leq^{S}_{S}\alpha$, $(\mathbb{F},V)\Vdash\beta\leq^{S}_{S}\neg \nomk$ and $(\mathbb{F},V)\Vdash\nomj\to\neg\nomk\leq^{S}_{S'}\neg\nomi$, then $V(\nomj\to\neg\nomk)\subseteq V(\neg\nomi)$, so $V(\nomi)\subseteq V(\nomj\land\nomk)$, since $\nomi,\nomj,\nomk$ are nominals, there interpretations are singletons, so $V(\nomi)=V(\nomj)=V(\nomk)$. Now from $(\mathbb{F},V)\Vdash\nomj\leq^{S}_{S}\alpha$ we have that $(W,(R_0\setminus S),V), V(\nomj)\Vdash\alpha$, and from $(\mathbb{F},V)\Vdash\beta\leq^{S}_{S}\neg \nomk$ we have that $(W,(R_0\setminus S),V), V(\nomk)\Vdash\neg\beta$, so $(W,(R_0\setminus S),V), V(\nomi)\Vdash\alpha$ and $(W,(R_0\setminus S),V), V(\nomi)\Vdash\neg\beta$, so $(\mathbb{F},V)\Vdash\alpha\to\beta\leq^{S}_{S'}\neg\nomi$.
\end{enumerate}
\end{proof}

\begin{lemma}\label{Lemma:residuation:neg:substage12}
The residuation rules for $\neg$ in Substage 1 and 2 are sound in $\mathbb{F}$.
\end{lemma}

\begin{proof}
It is easy to see that the residuation rules for $\neg$ in Substage 1 are special cases of the residuation rules for $\neg$ in Substage 2 (modulo double negation elimination). Now we only prove it for the residuation rule in Substage 2 where negation symbols occur on the right-hand side of the inequalities, the other rule is similar. 

For the soundness of the residuation rule for $\neg$, it suffices to show that for any Kripke frame $\mathbb{F}=(W,R_0)$, any valuation $V$ on $\mathbb{F}$, $(\mathbb{F},V)\Vdash\alpha\leq^{S}_{S'}\neg\beta$ iff $(\mathbb{F},V)\Vdash\beta\leq^{S'}_{S}\neg\alpha$. Indeed, it follows from the following equivalence:
\begin{center}
\begin{tabular}{c l}
& $(\mathbb{F},V)\Vdash\alpha\leq^{S}_{S'}\neg\beta$\\

iff & for all $w\in W$, if $(W, (R_0\setminus S), V), w\Vdash\alpha$, then $(W, (R_0\setminus S'), V), w\nVdash\beta$\\

iff & for all $w\in W$, if $(W, (R_0\setminus S'), V), w\Vdash\beta$, then $(W, (R_0\setminus S), V), w\nVdash\alpha$\\

iff & $(\mathbb{F},V)\Vdash\beta\leq^{S'}_{S}\neg\alpha$.\\
\end{tabular}
\end{center}
\end{proof}

\begin{proposition}\label{Prop:Substage:2}
The splitting rules, the residuation rules for $\neg,\Diamond,\Box,\Diamondblack,\blacksquare$, the second splitting rule in Substage 2 are sound in $\mathbb{F}$.
\end{proposition}

\begin{proof}
By Lemma \ref{Lemma:Splitting:substage12}, \ref{Lemma:residuation:neg:substage12}, \ref{Lemma:Residuation:white:substage2}, \ref{Lemma:Residuation:black:substage2} and \ref{Lemma:2ndsplitting:substage2}. 
\end{proof}

\begin{lemma}\label{Lemma:Residuation:white:substage2}
The residuation rules for $\Diamond,\Box$ in Substage 2 are sound in $\mathbb{F}$. 
\end{lemma}

\begin{proof}
We prove it for $\Diamond$, and the rule for $\Box$ is similar.

To show the soundness of the residuation rule for $\Diamond$ in Substage 2, it suffices to show that for any Kripke frame $\mathbb{F}=(W,R_0)$, any valuation $V$ on $\mathbb{F}$, $(\mathbb{F},V)\Vdash\Diamond\alpha\leq^{S}_{S'}\beta$ iff $(\mathbb{F},V)\Vdash\alpha\leq^{S}_{S'}(\Box^{S})^{-1}\beta$.

$\Rightarrow$: if $(\mathbb{F},V)\Vdash\Diamond\alpha\leq^{S}_{S'}\beta$, then for all $w\in W$, if $(W, (R_0\setminus S), V), w\Vdash\Diamond\alpha$, then $(W, (R_0\setminus S'), V), w\Vdash\beta$. Our aim is to show that for all $v\in W$, if $(W, (R_0\setminus S), V), v\Vdash\alpha$, then $(W, (R_0\setminus S'), V), v\Vdash(\Box^{S})^{-1}\beta$.

Consider any $v\in W$ such that $(W, (R_0\setminus S), V), v\Vdash\alpha$. Then for any $u\in W$ such that $(u,v)\in (R_0\setminus S)$, $(W, (R_0\setminus S), V), u\Vdash\Diamond\alpha$. Since $(\mathbb{F},V)\Vdash\Diamond\alpha\leq^{S}_{S'}\beta$, we have that $(W, (R_0\setminus S'), V), u\Vdash\beta$, so for any $u\in W$ such that $(v,u)\in (R_0\setminus S)^{-1}$, $(W, (R_0\setminus S'), V), u\Vdash\beta$, so $(W, (R_0\setminus S'), V), v\Vdash(\Box^{S})^{-1}\beta$.

$\Leftarrow$: if $(\mathbb{F},V)\Vdash\alpha\leq^{S}_{S'}(\Box^{S})^{-1}\beta$, then for all $w\in W$, if $(W, (R_0\setminus S), V), w\Vdash\alpha$, then $(W, (R_0\setminus S'), V), w\Vdash(\Box^{S})^{-1}\beta$. Our aim is to show that for all $v\in W$, if $(W, (R_0\setminus S), V), v\Vdash\Diamond\alpha$, then $(W, (R_0\setminus S'), V), v\Vdash\beta$.

Now assume that $(W, (R_0\setminus S), V), v\Vdash\Diamond\alpha$. Then there is a $u\in W$ such that $(v,u)\in (R_0\setminus S)$ and $(W, (R_0\setminus S), V), u\Vdash\alpha$. By $(\mathbb{F},V)\Vdash\alpha\leq^{S}_{S'}(\Box^{S})^{-1}\beta$, we have that $(W, (R_0\setminus S'), V), u\Vdash(\Box^{S})^{-1}\beta$. Therefore, for $v\in W$, we have $(u,v)\in (R_0\setminus S)^{-1}$, thus $(W, (R_0\setminus S'), V), v\Vdash\beta$.
\end{proof}

\begin{lemma}\label{Lemma:Residuation:black:substage2}
The residuation rules for $\Diamondblack,\blacksquare$ in Substage 2 are sound in $\mathbb{F}$. 
\end{lemma}

\begin{proof}
We prove it for $\blacksquare$, and the rule for $\Diamondblack$ is similar.

For the residuation rule for $\blacksquare$, it suffices to show that for any Kripke frame $\mathbb{F}=(W, R_0)$, any valuation $V$ on $\mathbb{F}$, $(\mathbb{F},V)\Vdash\alpha\leq^{S}_{S'}\blacksquare\beta$ iff $(\mathbb{F},V)\Vdash\forall\nomi_{m0}\forall\nomi_{m1}(\nomi_{m0}\leq^{S'}_{S'}\Diamond^{S'}\nomi_{m1}\ \Rightarrow\ \alpha \leq^{S}_{S'\cup\{(\nomi_{m0},\nomi_{m1})\}} \beta)$. Indeed, 

$(\mathbb{F},V)\Vdash\forall\nomi_{m0}\forall\nomi_{m1}(\nomi_{m0}\leq^{S'}_{S'}\Diamond^{S'}\nomi_{m1}\ \Rightarrow\ \alpha \leq^{S}_{S'\cup\{(\nomi_{m0},\nomi_{m1})\}} \beta)$,

iff for all $w,v\in W$, if $(w,v)\in (R_0\setminus S')$ then $(W,R_0,V^{\nomi_{m0},\nomi_{m1}}_{w,v})\Vdash\alpha \leq^{S}_{S'\cup\{(\nomi_{m0},\nomi_{m1})\}} \beta$, where $V^{\nomi_{m0},\nomi_{m1}}_{w,v}$ is the same as $V$ except that $V^{\nomi_{m0},\nomi_{m1}}_{w,v}(\nomi_{m0})=\{w\}$, $V^{\nomi_{m0},\nomi_{m1}}_{w,v}(\nomi_{m1})=\{v\}$, 

iff for all $u,v,w\in W$, if $(w,v)\in (R_0\setminus S')$ and $(W,(R_0\setminus S),V^{\nomi_{m0},\nomi_{m1}}_{w,v}),u\Vdash\alpha$, then $(W,(R_0\setminus (S'\cup\{(\nomi_{m0},\nomi_{m1})\})),V^{\nomi_{m0},\nomi_{m1}}_{w,v}),u\Vdash\beta$, 

iff for all $u\in W$, if $(W, (R_0\setminus S),V^{\nomi_{m0},\nomi_{m1}}_{w,v}),u\Vdash\alpha$, then for all $v,w\in W$, if $(w,v)\in (R_0\setminus S')$ then $(W,(R_0\setminus (S'\cup\{(\nomi_{m0},\nomi_{m1})\})),V^{\nomi_{m0},\nomi_{m1}}_{w,v}),u\Vdash\beta$, 

iff for all $u\in W$, if $(W, (R_0\setminus S),V^{\nomi_{m0},\nomi_{m1}}_{w,v}),u\Vdash\alpha$, then $(W, (R_0\setminus S'),V^{\nomi_{m0},\nomi_{m1}}_{w,v}),u\Vdash\blacksquare\beta$,

iff for all $u\in W$, if $(W, (R_0\setminus S),V),u\Vdash\alpha$, then  $(W, (R_0\setminus S'),V),u\Vdash\blacksquare\beta$ (since $\nomi_{m0}$ and $\nomi_{m1}$ do not occur in $\alpha$ and $\beta$),

iff $(\mathbb{F},V)\Vdash\alpha\leq^{S}_{S'}\blacksquare\beta$.
\end{proof}

\begin{lemma}\label{Lemma:2ndsplitting:substage2}
The second splitting rule in Substage 2 is sound in $\mathbb{F}$.
\end{lemma}

\begin{proof}
It follows immediately from the meta-equivalence that $\forall x\forall y(\alpha\rightarrow\beta\land\gamma)\leftrightarrow\forall x\forall y(\alpha\rightarrow\beta)\land\forall x\forall y(\alpha\rightarrow\gamma)$.
\end{proof}

\begin{proposition}\label{Prop:Substage:3}
The packing rules in Substage 3 are sound in $\mathbb{F}$.
\end{proposition}

\begin{proof}
We only prove the soundness of the first packing rule, the others are similar.

It is easy to see that in the mega-inequality of the premise and in the conclusion, contextual connectives $\Box,\Diamond,\blacksquare,\Diamondblack$ do not occur, so we can ignore the superscripts and subscripts in the inequalities occuring in the rule. 

Therefore, for any Kripke frame $\mathbb{F}=(W,R_0)$ and any valuation $V$ on it, 

$(\mathbb{F},V)\Vdash\forall\nomi_{m_{0}0}\forall\nomi_{m_{0}1}
(\nomi_{m_{0}0}\leq^{S_{0}}_{S_{0}}\Diamond^{S_{0}}\nomi_{m_{0}1}\Rightarrow\ldots
\forall\nomi_{m_{k}0}\forall\nomi_{m_{k}1}(\nomi_{m_{k}0}\leq^{S_{k}}_{S_{k}}\Diamond^{S_{k}}\nomi_{m_{k}1}\ \Rightarrow\ \alpha\leq^{S}_{S'} p)\ldots)$,

iff for all $w_{m_0 0}, w_{m_0 1}, \ldots, w_{m_k 0}, w_{m_k 1}\in W$, if $(w_{m_0 0},w_{m_0 1})\in (R_0\setminus S_0)$, \ldots, $(w_{m_k 0},w_{m_k 1})\in (R_0\setminus S_k)$, then $(\mathbb{F}, V^{\nomi_{m_0 0}, \nomi_{m_0 1}, \ldots, \nomi_{m_k 0}, \nomi_{m_k 1}}_{w_{m_0 0}, w_{m_0 1}, \ldots, w_{m_k 0}, w_{m_k 1}})\Vdash\alpha\leq^{S}_{S'} p$,

iff for all $w_{m_0 0}, w_{m_0 1}, \ldots, w_{m_k 0}, w_{m_k 1}\in W$, if $(w_{m_0 0},w_{m_0 1})\in (R_0\setminus S_0)$, \ldots, $(w_{m_k 0},w_{m_k 1})\in (R_0\setminus S_k)$, then $(\mathbb{F}, V^{\nomi_{m_0 0}, \nomi_{m_0 1}, \ldots, \nomi_{m_k 0}, \nomi_{m_k 1}}_{w_{m_0 0}, w_{m_0 1}, \ldots, w_{m_k 0}, w_{m_k 1}})\Vdash\alpha\leq^{\emptyset}_{\emptyset} p$,

iff for all $w_{m_0 0}, w_{m_0 1}, \ldots, w_{m_k 0}, w_{m_k 1}, v\in W$, if $(w_{m_0 0},w_{m_0 1})\in (R_0\setminus S_0)$, \ldots, $(w_{m_k 0},w_{m_k 1})\in (R_0\setminus S_k)$, $(\mathbb{F}, V^{\nomi_{m_0 0}, \nomi_{m_0 1}, \ldots, \nomi_{m_k 0}, \nomi_{m_k 1}}_{w_{m_0 0}, w_{m_0 1}, \ldots, w_{m_k 0}, w_{m_k 1}}),v\Vdash\alpha$, then $(\mathbb{F}, V^{\nomi_{m_0 0}, \nomi_{m_0 1}, \ldots, \nomi_{m_k 0}, \nomi_{m_k 1}}_{w_{m_0 0}, w_{m_0 1}, \ldots, w_{m_k 0}, w_{m_k 1}}),v\Vdash p$,

iff for all $w_{m_0 0}, w_{m_0 1}, \ldots, w_{m_k 0}, w_{m_k 1}, v\in W$, if $(w_{m_0 0},w_{m_0 1})\in (R_0\setminus S_0)$, \ldots, $(w_{m_k 0},w_{m_k 1})\in (R_0\setminus S_k)$, $(\mathbb{F}, V^{\nomi_{m_0 0}, \nomi_{m_0 1}, \ldots, \nomi_{m_k 0}, \nomi_{m_k 1}}_{w_{m_0 0}, w_{m_0 1}, \ldots, w_{m_k 0}, w_{m_k 1}}),v\Vdash\alpha$, then $(\mathbb{F}, V),v\Vdash p$,

iff for all $w_{m_0 0}, w_{m_0 1}, \ldots, w_{m_k 0}, w_{m_k 1}, v\in W$, if $(\mathbb{F}, V^{\nomi_{m_0 0}, \nomi_{m_0 1}, \ldots, \nomi_{m_k 0}, \nomi_{m_k 1}}_{w_{m_0 0}, w_{m_0 1}, \ldots, w_{m_k 0}, w_{m_k 1}})\Vdash A(\nomi_{m_0 0}\to\nomi_{m_0 1}), \ldots, A(\nomi_{m_k 0}\to\nomi_{m_k 1})$, $(\mathbb{F}, V^{\nomi_{m_0 0}, \nomi_{m_0 1}, \ldots, \nomi_{m_k 0}, \nomi_{m_k 1}}_{w_{m_0 0}, w_{m_0 1}, \ldots, w_{m_k 0}, w_{m_k 1}}),v\Vdash\alpha$, then $(\mathbb{F}, V),v\Vdash p$,

iff for all $w_{m_0 0}, w_{m_0 1}, \ldots, w_{m_k 0}, w_{m_k 1}, v\in W$, if $(\mathbb{F}, V^{\nomi_{m_0 0}, \nomi_{m_0 1}, \ldots, \nomi_{m_k 0}, \nomi_{m_k 1}}_{w_{m_0 0}, w_{m_0 1}, \ldots, w_{m_k 0}, w_{m_k 1}}),v\Vdash A(\nomi_{m_0 0}\to\nomi_{m_0 1})\land\ldots\land A(\nomi_{m_k 0}\to\nomi_{m_k 1})\land\alpha$, then $(\mathbb{F}, V),v\Vdash p$,

iff for all $v\in W$, if there exists $w_{m_0 0}, w_{m_0 1}, \ldots, w_{m_k 0}, w_{m_k 1}\in W$ such that $(\mathbb{F}, V^{\nomi_{m_0 0}, \nomi_{m_0 1}, \ldots, \nomi_{m_k 0}, \nomi_{m_k 1}}_{w_{m_0 0}, w_{m_0 1}, \ldots, w_{m_k 0}, w_{m_k 1}}),v\Vdash A(\nomi_{m_0 0}\to\nomi_{m_0 1})\land\ldots\land A(\nomi_{m_k 0}\to\nomi_{m_k 1})\land\alpha$, then $(\mathbb{F}, V),v\Vdash p$,

iff for all $v\in W$, if $(\mathbb{F}, V),v\Vdash\exists\nomi_{m_{0}0}\exists\nomi_{m_{0}1}\ldots\exists\nomi_{m_{k}0}\exists\nomi_{m_{k}1}(A(\nomi_{m_0 0}\to\nomi_{m_0 1})\land\ldots\land A(\nomi_{m_k 0}\to\nomi_{m_k 1})\land\alpha)$, then $(\mathbb{F}, V),v\Vdash p$,

iff $(\mathbb{F}, V)\Vdash\exists\nomi_{m_{0}0}\exists\nomi_{m_{0}1}\ldots\exists\nomi_{m_{k}0}\exists\nomi_{m_{k}1}(
A(\nomi_{m_{0}0}\to\Diamond^{S_{0}}\nomi_{m_{0}1})
\land
A(\nomi_{m_{k}0}\to\Diamond^{S_{k}}\nomi_{m_{k}1})
\land\alpha)\leq^{\emptyset}_{\emptyset} p$.
\end{proof}

\begin{proposition}\label{Prop:Substage:4}
The Ackermann rules in Substage 4 are sound in $\mathbb{F}$.
\end{proposition}

\begin{proof}
We only prove it for the right-handed Ackermann rule, the other rule is similar.

Without loss of generality we assume that $n=m=1$. It suffices to show the following right-handed Ackermann lemma:
\begin{lemma}\label{Lemma:Right:Ackermann}
Assume $\alpha$ is pure and contains no contextual modalities $\Box,\Diamond,\blacksquare,\Diamondblack$ and does not contain nominals in $\vec{\nomi}$, $\beta$ is positive in $p$ and $\gamma$ is negative in $p$, then for any Kripke frame $\mathbb{F}=(W,R_0)$ and any valuation $V$ on it, 

$(\mathbb{F},V)\Vdash\forall\vec{\nomi}(\beta(\alpha/p)\leq^{T}_{T'}\gamma(\alpha/p))$\ \ \  iff\ \ \  there exists a valuation $V^p$ such that $(\mathbb{F},V^{p})\Vdash\alpha\leq^{\emptyset}_{\emptyset}p$ and $(\mathbb{F},V^{p})\Vdash\forall\vec{\nomi}(\beta\leq^{T}_{T'}\gamma)$, where $V^{p}$ is the same as $V$ except $V^{p}(p)$.
\end{lemma}

Notice that $\alpha$ and $p$ do not contain contextual modalities, so their valuation do not change when the context is different.

$\Rightarrow$: Take $V^{p}$ such that $V^{p}$ is the same as $V$ except that $V^{p}(p)=V(\alpha)$. Since $\alpha$ does not contain $p$, it is easy to see that $V^{p}(\alpha)=V(\alpha)=V^{p}(p)$. Therefore $(\mathbb{F},V^{p})\Vdash\alpha\leq^{\emptyset}_{\emptyset}p$. Since the valuation of $\alpha$ and $p$ do not change when the context is different, so for any $w\in W$, 

$(W, (R_0\setminus T),V^{p}), w\Vdash \beta$ iff $(W, (R_0\setminus T),V), w\Vdash \beta(\alpha/p)$, and 

$(W, (R_0\setminus T'),V^{p}), w\Vdash \gamma$ iff $(W, (R_0\setminus T'),V), w\Vdash \gamma(\alpha/p)$, so 

from $(\mathbb{F},V)\Vdash\forall\vec{\nomi}(\beta(\alpha/p)\leq^{T}_{T'}\gamma(\alpha/p))$ one can get $(\mathbb{F},V^{p})\Vdash\forall\vec{\nomi}(\beta\leq^{T}_{T'}\gamma)$.

$\Leftarrow$: This direction follows from the monotonicity of $\beta$ in $p$ and the antitonicity of $\gamma$ in $p$, and that the valuation of $\alpha$ and $p$ do not change when the context is different.
\end{proof}

\subsection{Success of $\mathsf{ALBA}^{\mathsf{SML}}$ on Sahlqvist inequalities}\label{Subsec:Success}

In the present subsection we show that $\mathsf{ALBA}^{\mathsf{SML}}$ succeeds on all Sahlqvist inequalities:

\begin{theorem}\label{Thm:Success}
$\mathsf{ALBA}^{\mathsf{SML}}$ succeeds on all Sahlqvist inequalities.
\end{theorem}

\begin{definition}[Definite $\epsilon$-Sahlqvist inequality]
Given an order type $\epsilon$, $*\in\{-,+\}$, the signed generation tree $*\phi$ of the term $\phi(p_1,\ldots, p_n)$ is \emph{definite $\epsilon$-Sahlqvist} if there is no $+\lor,-\land$ occurring in the outer part on an $\epsilon$-critical branch. An inequality $\phi\leq\psi$ is definite $\epsilon$-Sahlqvist if the trees $+\phi$ and $-\psi$ are both definite $\epsilon$-Sahlqvist.
\end{definition}

\begin{lemma}\label{Lemma:Stage:1}
Let $\{\phi_i\leq\psi_i\}_{i\in I}=\mathsf{Preprocess}(\phi\leq\psi)$ obtained by exhaustive application of the rules in Stage 1 on an input $\epsilon$-Sahlqvist inequality $\phi\leq\psi$. Then each $\phi_i\leq\psi_i$ is a definite $\epsilon$-Sahlqvist inequality.
\end{lemma}

\begin{proof}
It is easy to see that by applying the distribution rules, all occurrences of $+\lor$ and $-\land$ in the outer part of an $\epsilon$-critical branch have been pushed up towards the root of the signed generation trees $+\phi$ and $-\psi$. Then by exhaustively applying the splitting rules, all such $+\lor$ and $-\land$ are eliminated. Since by applying the distribution rules, the splitting rules and the monotone/antitone variable elimination rules do not change the $\epsilon$-Sahlqvistness of a signed generation tree, in $\mathsf{Preprocess}(\phi\leq\psi)$, each signed generation tree $+\phi_i$ and $-\psi_i$ are $\epsilon$-Sahlqvist, and since they do not have $+\lor$ and $-\land$ in the outer part in the $\epsilon$-critical branches, they are definite.
\end{proof}

\begin{definition}[Inner $\epsilon$-Sahlqvist signed generation tree]
Given an order type $\epsilon$, $*\in\{-,+\}$, the signed generation tree $*\phi$ of the term $\phi(p_1,\ldots, p_n)$ is \emph{inner $\epsilon$-Sahlqvist} if its outer part $P_2$ on an $\epsilon$-critical branch is always empty, i.e.\ its $\epsilon$-critical branches have inner nodes only.
\end{definition}

\begin{lemma}\label{Lemma:Substage:2:1}
Given inequalities $\nomi_0\leq^{\emptyset}_{\emptyset}\phi_i$ and $\psi_i\leq^{\emptyset}_{\emptyset}\neg\nomi_1$obtained from Stage 1 where $+\phi_i$ and $-\psi_i$ are definite $\epsilon$-Sahlqvist, by applying the rules in Substage 1 of Stage 2 exhaustively, the inequalities that we get are in one of the following forms:

\begin{enumerate}
    \item pure inequalities which does not have occurrences of propositional variables;
    \item inequalities of the form $\nomi\leq^{S}_{S'}\alpha$ where $+\alpha$ is inner $\epsilon$-Sahlqvist;
    \item inequalities of the form $\beta\leq^{S}_{S'}\neg\nomi$ where $-\beta$ is inner $\epsilon$-Sahlqvist.
\end{enumerate}
\end{lemma}

\begin{proof}
Indeed, the rules in the Substage 1 of Stage 2 deal with outer nodes in the signed generation trees $+\phi_i$ and $-\psi_i$ except $+\lor$,$-\land$. For each rule, without loss of generality assume we start with an inequality of the form $\nomi\leq^{S}_{S'}\alpha$, then by applying the rules in Substage 1 of Stage 2, the inequalities we get are either a pure inequality without propositional variables, or 
an inequality where the left-hand side (resp.\ right-hand side) is $\nomi$ (resp.\ $\neg\nomi$), and the other side is a formula $\alpha'$ which is a subformula of $\alpha$, such that $\alpha'$ has one root connective less than $\alpha$. Indeed, if $\alpha'$ is on the left-hand side (resp.\ right-hand side) then $-\alpha'$ ($+\alpha'$) is definite $\epsilon$-Sahlqvist.

By applying the rules in the Substage 1 of Stage 2 exhaustively, we can eliminate all the outer connectives in the critical branches, so for non-pure inequalities, they become of form 2 or form 3.
\end{proof}

\begin{lemma}\label{Lemma:Substage:2:2}
Assume we have an inequality $\nomi\leq^{S}_{S'}\alpha$ or $\beta\leq^{S}_{S'}\neg\nomi$ where $+\alpha$ and $-\beta$ are inner $\epsilon$-Sahlqvist, by applying the rules in Substage 2 of Stage 2, we have (mega-)inequalities ($k$ can be 0 where a mega-inequality becomes an inequality) of the following form:

\begin{enumerate}
\item $\forall\nomi_{m_{0}0}\forall\nomi_{m_{0}1}(\nomi_{m_{0}0}\leq^{S_{0}}_{S_{0}}\Diamond^{S_{0}}\nomi_{m_{0}1}\Rightarrow\ldots\forall\nomi_{m_{k}0}\forall\nomi_{m_{k}1}(\nomi_{m_{k}0}\leq^{S_{k}}_{S_{k}}\Diamond^{S_{k}}\nomi_{m_{k}1}\Rightarrow\alpha\leq^{S}_{S'}p)\ldots)$,
where $\epsilon(p)=1$, $\alpha$ is pure and does not contain contextual connectives $\Box,\Diamond,\blacksquare,\Diamondblack$;

\item $\forall\nomi_{m_{0}0}\forall\nomi_{m_{0}1}(\nomi_{m_{0}0}\leq^{S_{0}}_{S_{0}}\Diamond^{S_{0}}\nomi_{m_{0}1}\Rightarrow\ldots\forall\nomi_{m_{k}0}\forall\nomi_{m_{k}1}(\nomi_{m_{k}0}\leq^{S_{k}}_{S_{k}}\Diamond^{S_{k}}\nomi_{m_{k}1}\Rightarrow p\leq^{S}_{S'}\beta)\ldots)$,
where $\epsilon(p)=\partial$, $\beta$ is pure and does not contain contextual connectives $\Box,\Diamond,\blacksquare,\Diamondblack$;

\item $\forall\nomi_{m_{0}0}\forall\nomi_{m_{0}1}(\nomi_{m_{0}0}\leq^{S_{0}}_{S_{0}}\Diamond^{S_{0}}\nomi_{m_{0}1}\Rightarrow\ldots\forall\nomi_{m_{k}0}\forall\nomi_{m_{k}1}(\nomi_{m_{k}0}\leq^{S_{k}}_{S_{k}}\Diamond^{S_{k}}\nomi_{m_{k}1}\Rightarrow\alpha\leq^{S}_{S'}\gamma)\ldots)$,
where $\alpha$ is pure and does not contain contextual connectives $\Box,\Diamond,\blacksquare,\Diamondblack$, and $+\gamma$ is $\epsilon^{\partial}$-uniform;

\item $\forall\nomi_{m_{0}0}\forall\nomi_{m_{0}1}(\nomi_{m_{0}0}\leq^{S_{0}}_{S_{0}}\Diamond^{S_{0}}\nomi_{m_{0}1}\Rightarrow\ldots\forall\nomi_{m_{k}0}\forall\nomi_{m_{k}1}(\nomi_{m_{k}0}\leq^{S_{k}}_{S_{k}}\Diamond^{S_{k}}\nomi_{m_{k}1}\Rightarrow\gamma\leq^{S}_{S'}\beta)\ldots)$,
where $\beta$ is pure and does not contain contextual connectives $\Box,\Diamond,\blacksquare,\Diamondblack$, and $-\gamma$ is $\epsilon^{\partial}$-uniform.
\end{enumerate}
\end{lemma}

\begin{proof}
First of all, from the rules of the Substage 2 of Stage 2, it is easy to see that from the given inequality, what we will obtain would be a set of mega-inequalities, and by the second splitting rule those mega-inequalities are built up from inequalities by $\forall\nomi\forall\nomj(\nomi\leq^{S}_{S}\Diamond^{S}\nomj\Rightarrow \mathsf{Mega})$, so we will have mega-inequalities of the form $\forall\nomi_{m_{0}0}\forall\nomi_{m_{0}1}(\nomi_{m_{0}0}\leq^{S_{0}}_{S_{0}}\Diamond^{S_{0}}\nomi_{m_{0}1}\Rightarrow\ldots\forall\nomi_{m_{k}0}\forall\nomi_{m_{k}1}(\nomi_{m_{k}0}\leq^{S_{k}}_{S_{k}}\Diamond^{S_{k}}\nomi_{m_{k}1}\Rightarrow\gamma\leq^{S}_{S'}\delta)\ldots)$. Now it suffices to check the shape of $\gamma$ and $\delta$. (From now on we call $\gamma\leq^{S}_{S'}\delta$ the \emph{head} of the mega-inequality.)

Notice that for each input inequality, it is of the form $\nomi\leq^{S}_{S'}\alpha$ or $\beta\leq^{S}_{S'}\neg\nomi$, where $+\alpha$ and $-\beta$ are inner $\epsilon$-Sahlqvist. By applying the splitting rules and the residuation rules in this substage, it is easy to see that the head of the (mega-)inequality will have one side of the inequality pure and have no contextual connectives $\Box, \Diamond, \blacksquare, \Diamondblack$, and the other side still inner $\epsilon$-Sahlqvist. By applying these rules exhaustively, one will either have $p$ as the non-pure side (with this $p$ on a critical branch), or have an inner $\epsilon$-Sahlqvist signed generation tree with no critical branch, i.e.,\ $\epsilon^{\partial}$-uniform.
\end{proof}

\begin{lemma}\label{Lemma:Substage:2:3}
Assume we have ($mega$-)inequalities of the form as described in Lemma \ref{Lemma:Substage:2:2}. Then we can get inequalities of the following form:

\begin{enumerate}
\item $\alpha\leq^{\emptyset}_{\emptyset}p$ where $\epsilon(p)=1$, $\alpha$ is pure and do not contain contextual connectives $\Box,\Diamond,\blacksquare,\Diamondblack$;
\item $p\leq^{\emptyset}_{\emptyset}\alpha$ where $\epsilon(p)=\partial$, $\alpha$ is pure and do not contain contextual connectives $\Box,\Diamond,\blacksquare,\Diamondblack$;
\item $\forall\nomi_1\ldots\forall\nomi_n(\top\leq^{\emptyset}_{S}\gamma)$ where $+\gamma$ is $\epsilon^{\partial}$-uniform.
\end{enumerate} 
\end{lemma}

\begin{proof}
From the shape of the mega-inequalities, we can see that for all the mega-inequalities we can apply the corresponding packing rules so that we can get the inequalities as described in the lemma.
\end{proof}

\begin{lemma}\label{Lemma:Substage:2:4}
Assume we have inequalities of the form as described in Lemma \ref{Lemma:Substage:2:3}, the Ackermann lemmas are applicable and therefore all propositional variables can be eliminated.
\end{lemma}

\begin{proof}
Immediate observation from the requirements of the Ackermann lemmas.
\end{proof}

\begin{proof}[Proof of Theorem \ref{Thm:Success}]
Assume we have an $\epsilon$-Sahlqvist inequality $\phi\leq\psi$ as input. By Lemma \ref{Lemma:Stage:1}, we get a set of definite  $\epsilon$-Sahlqvist inequalities. Then by Lemma \ref{Lemma:Substage:2:1}, we get inequalities as described in Lemma \ref{Lemma:Substage:2:1}. By Lemma \ref{Lemma:Substage:2:2}, we get the mega-inequalities as described. Therefore by Lemma \ref{Lemma:Substage:2:3}, we can apply the packing rules to get inequalities and universally quantified inequalities as described in the lemma. Finally by Lemma \ref{Lemma:Substage:2:4}, the (universally quantified) inequalities are in the right shape to apply the Ackermann rules, and thus we can eliminate all the propositional variables and the algorithm succeeds on the input.
\end{proof}

\section{Discussions and further directions}\label{Sec:discussion}
\subsection{Completeness and canonicity issues}

One issue that is not treated in the present paper is the completeness and canonicity theory of sabotage modal logic. For completeness theory, as is discussed in \cite{AuvBGr18}, there is not yet natural Hilbert-style axiomatization system for the basic sabotage modal logic, and there is no canonical model for sabotage modal logic, so it is not yet ready to provide a Sahlqvist-type completeness theorem for the current state-of-art. Regarding canonicity theory, one possible approach is to reduce the canonicity problem for formulas in sabotage modal logic to canonicity problems for formulas in hybrid logic with binders, since sabotage modal logic is a fragment of hybrid logic with binders (for more details, see \cite{ArFeHoMa16}), for which the Sahlqvist theory is still in development (see \cite{Zh21}). For the algebraic approach of canonicity, one needs to first define the algebraic semantics for sabotage modal logic.

\subsection{Other questions}

Other interesting questions include the following:

\begin{itemize}
\item Extending the Sahlqvist sabotage formulas to inductive sabotage formulas as well as to the language of sabotage modal logic with fixpoint operators;
\item A Kracht-type theorem characterizing the first-order correspondents of Sahlqvist sabotage formulas; 
\item A Goldblatt-Thomason-type theorem characterizing the sabotage modally definable classes of Kripke frames;
\item Extend results on sabotage modal logic to the class of relation changing modal logics \cite{ArFeHoMa16};
\item Since sabotage modal logic can be treated as a fragment of very expressive hybrid logics (see \cite{ArFeHoMa16}), therefore a relevant interesting question would be to have a Sahlqvist-type correspondence theory for very expressive hybrid logics. For other existing works of correspondence theory for hybrid logic, see \cite{Co09,CoGoVa06b,ConRob,GoVa00,Ta05,tCMaVi06}.
\end{itemize}

\paragraph{Acknowledgement} The research of the author is supported by Taishan University Starting Grant ``Studies on Algebraic Sahlqvist Theory'' and the Taishan Young Scholars Program of the Government of Shandong Province, China (No.tsqn201909151).

\bibliographystyle{abbrv}
\bibliography{Sabotage}

\begin{thebibliography}{10}

\bibitem{ArFeHo13}
C.~Areces, R.~Fervari, and G.~Hoffmann.
\newblock Swap logic.
\newblock {\em Logic Journal of IGPL}, 22(2):309--332, 2014.

\bibitem{ArFeHo15}
C.~Areces, R.~Fervari, and G.~Hoffmann.
\newblock Relation-changing modal operators.
\newblock {\em Logic Journal of the IGPL}, 23(4):601--627, 2015.

\bibitem{ArFeHoMa16}
C.~Areces, R.~Fervari, G.~Hoffmann, and M.~Martel.
\newblock Relation-changing logics as fragments of hybrid logics.
\newblock In D.~Cantone and G.~Delzanno, editors, {\em Proceedings of the
  Seventh International Symposium on Games, Automata, Logics and Formal
  Verification}, volume 226 of {\em Electronic Proceedings in Theoretical
  Computer Science}, pages 16--29. Open Publishing Association, 2016.

\bibitem{AuBadCHe09}
G.~Aucher, P.~Balbiani, L.~F. del Cerro, and A.~Herzig.
\newblock Global and local graph modifiers.
\newblock {\em Electronic Notes in Theoretical Computer Science}, 231:293--307,
  2009.

\bibitem{AuvBGr18}
G.~Aucher, J.~Van~Benthem, and D.~Grossi.
\newblock Modal logic of sabotage revisited.
\newblock {\em Journal of Logic and Computation}, 28:269--303, 2018.

\bibitem{BeBlWo06}
P.~Blackburn, J.~F. van Benthem, and F.~Wolter.
\newblock {\em Handbook of modal logic}, volume~3.
\newblock Elsevier, 2006.

\bibitem{Co09}
W.~Conradie.
\newblock Completeness and correspondence in hybrid logic via an extension of
  {SQEMA}.
\newblock {\em Electronic Notes in Theoretical Computer Science}, 231:175--190,
  2009.

\bibitem{CoGoVa06}
W.~Conradie, V.~Goranko, and D.~Vakarelov.
\newblock {A}lgorithmic correspondence and completeness in modal logic. {I}.
  {T}he core algorithm {SQEMA}.
\newblock {\em Logical Methods in Computer Science}, 2:1--26, 2006.

\bibitem{CoGoVa06b}
W.~Conradie, V.~Goranko, and D.~Vakarelov.
\newblock Algorithmic correspondence and completeness in modal logic. {II}.
  polyadic and hybrid extensions of the algorithm {SQEMA}.
\newblock {\em Journal of Logic and Computation}, 16(5):579--612, 2006.

\bibitem{CoPa12}
W.~Conradie and A.~Palmigiano.
\newblock Algorithmic correspondence and canonicity for distributive modal
  logic.
\newblock {\em Annals of Pure and Applied Logic}, 163(3):338 -- 376, 2012.

\bibitem{CPSZ}
W.~Conradie, A.~Palmigiano, S.~Sourabh, and Z.~Zhao.
\newblock Canonicity and relativized canonicity via pseudo-correspondence: an
  application of {ALBA}.
\newblock Submitted. ArXiv preprint 1511.04271.

\bibitem{CPZ:Trans}
W.~Conradie, A.~Palmigiano, and Z.~Zhao.
\newblock Sahlqvist via translation.
\newblock {\em Logical Methods in Computer Science}, 15, 2019.

\bibitem{ConRob}
W.~Conradie and C.~Robinson.
\newblock On {S}ahlqvist theory for hybrid logic.
\newblock {\em Journal of Logic and Computation}, 27(3):867--900, 2017.

\bibitem{GoVa00}
V.~Goranko and D.~Vakarelov.
\newblock Sahlqvist formulas unleashed in polyadic modal languages.
\newblock {\em Advances in Modal Logic 3}, pages 221--240, 2000.

\bibitem{KoRe11}
B.~Kooi and B.~Renne.
\newblock Arrow update logic.
\newblock {\em The Review of Symbolic Logic}, 4(04):536--559, 2011.

\bibitem{Li18}
D.~Li.
\newblock Losing connection: the modal logic of definable link deletion.
\newblock {\em Journal of Logic and Computation}, 30(3):715--743, 2020.

\bibitem{LoRo03a}
C.~L{\"o}ding and P.~Rohde.
\newblock Model checking and satisfiability for sabotage modal logic.
\newblock In {\em International Conference on Foundations of Software
  Technology and Theoretical Computer Science}, pages 302--313. Springer, 2003.

\bibitem{LoRo03b}
C.~L{\"o}ding and P.~Rohde.
\newblock Solving the sabotage game is pspace-hard.
\newblock In {\em International Symposium on Mathematical Foundations of
  Computer Science}, pages 531--540. Springer, 2003.

\bibitem{PaSoZh16}
A.~Palmigiano, S.~Sourabh, and Z.~Zhao.
\newblock Sahlqvist theory for impossible worlds.
\newblock {\em Journal of Logic and Computation}, 27(3):775--816, 2017.

\bibitem{Sa75}
H.~Sahlqvist.
\newblock Completeness and correspondence in the first and second order
  semantics for modal logic.
\newblock In {\em Studies in Logic and the Foundations of Mathematics},
  volume~82, pages 110--143. 1975.

\bibitem{Ta05}
K.~Tamura.
\newblock Hybrid logic with pure and {S}ahlqvist axioms.
\newblock {\em http://www.st.nanzan-u.ac.jp/info/sasaki/2005mlg/43-45.pdf}.

\bibitem{tCMaVi06}
B.~ten Cate, M.~Marx, and J.~P. Viana.
\newblock Hybrid logics with {S}ahlqvist axioms.
\newblock {\em Logic Journal of the IGPL}, (3):293--300, 2006.

\bibitem{vB05}
J.~van Benthem.
\newblock An essay on sabotage and obstruction.
\newblock In {\em Mechanizing Mathematical Reasoning}, pages 268--276.
  Springer, 2005.

\bibitem{vBMiZB19}
J.~van Benthem, K.~Mierzewski, and F.~Zaffora~Blando.
\newblock The modal logic of stepwise removal.
\newblock {\em Reports on Symbolic Logic}, 2019.

\bibitem{Zh21}
Z.~Zhao.
\newblock Algorithmic correspondence for hybrid logic with binder.
\newblock {\em In preparation}.

\end{thebibliography}
\end{document}